\documentclass[12pt,reqno]{amsart}
 \usepackage{ifthen}
 \usepackage{amsmath}
 \usepackage{amssymb}
 \usepackage{amsthm}
 \usepackage[latin1]{inputenc}
 \usepackage{eurosym}
 \usepackage[dvips]{graphics}
 \usepackage{graphicx}
 \usepackage{epsfig}
\usepackage[colorlinks]{hyperref}
\usepackage{enumerate}
\theoremstyle{plain}
\newtheorem{theorem}{Theorem}
\newtheorem{proposition}{Proposition}
\usepackage{hyperref}

\newtheorem{corollary}{Corollary}
\newtheorem{lemma}{Lemma}

\theoremstyle{definition}
\newtheorem{definition}{Definition}

\newtheorem{remark}{Remark}

\newcommand{\C}{\mathbb C}
\newcommand{\E}{\mathbb E}

\newcommand{\N}{\mathbb N}

\newcommand{\R}{\mathbb R}
\newcommand{\T}{\mathbb T}
\newcommand{\V}{\mathbb V}
\newcommand{\Z}{\mathbb Z}

\newcommand{\cC}{\mathcal C}
\newcommand{\cI}{\mathcal I}

\newcommand{\cH}{\mathcal H}

\newcommand{\cL}{\mathcal L}

\newcommand{\cN}{\mathcal N}
\newcommand{\cP}{\mathcal P}

\newcommand{\bF}{\mathbf F}
\newcommand{\fui}{\varphi}
\newcommand{\af}{\alpha}

\newcommand{\De}{\Delta}
\newcommand{\Ga}{\Gamma}
\newcommand{\lam}{\lambda}
\newcommand{\Lam}{\Lambda}
\newcommand{\te}{\theta}
\newcommand{\om}{\omega}
\newcommand{\ep}{\varepsilon}

\newcommand{\diver}{\operatorname{div}}

\newcommand{\lV}{\left\Vert}
\newcommand{\rV}{\right\Vert}

\newcommand{\rip}{\rangle}
\newcommand{\lip}{\langle}

\newtheorem{problem}{Problem}

\begin{document}
\title[Discrete approximation of MFG]{Discrete approximation of
  stationary Mean Field Games}
\author{Tigran Bakaryan}
\address[T. Bakaryan]{
	King Abdullah University of Science and Technology (KAUST),
	CEMSE Division, Thuwal 23955-6900, Saudi Arabia.}
\email{tigran.bakaryan@kaust.edu.sa}

\author{Diogo Gomes}
\address[D. Gomes]{
	King Abdullah University of Science and Technology (KAUST),
	CEMSE Division, Thuwal 23955-6900, Saudi Arabia.}
\email{diogo.gomes@kaust.edu.sa}

\author{H\'ector S\'anchez Morgado}
\address[H. S\'anchez]{Instituto de Matem\'aticas. Universidad Nacional Aut\'onoma de M\'exico}
\email{hector@matem.unam.mx}

\date{\today}
   \begin{abstract}
In this paper, we focus on stationary (ergodic) mean-field games (MFGs). These games arise in the study of the long-time behavior of finite-horizon MFGs. 
Motivated by a prior scheme for Hamilton--Jacobi equations introduced in Aubry--Mather's theory, we introduce a discrete approximation to stationary MFGs. Relying on Kakutani's fixed-point theorem, we prove the existence and uniqueness (up to additive constant) of solutions to the discrete problem. Moreover, we show that the solutions to the discrete problem converge, uniformly in the nonlocal case and weakly in the local case, to the classical solutions of the stationary problem.
\end{abstract}

\maketitle

\section{Introduction}
\label{sec:dMFG}

The mean-field games (MFG) theory aims to model and analyze systems with many competing rational agents. 
These agents have preferences encoded in a cost functional, which they
seek to optimize. This functional depends both on the agent's and on the other agent's states. Individually, agents have
little effect on the entire system. Accordingly, their cost depends only
on their state and on aggregate or statistical quantities.   MFGs theory was introduced
in the mathematics community by J-M. Lasry
and P-L. Lions in \cite{ll1,ll2} and independently in the engineering
community by M. Huang, P. Caines, and R. Malham{\'e} in
\cite{Caines1,Caines2}. This theory has expanded tremendously and has
found applications in population dynamics  \cite{Lachapelle10,
  Achdou2019,GomesMS19},  economics \cite{NBERw23732,GNP,GoGuRi2021},
finance \cite{cardaliaguet2017mean,cartea2015algorithmic}, engineering
\cite{DePaola2019,KizilkaleSM19}, to name just a few. 


An important  model in the theory of MFG is the stationary (ergodic) MFG. This problem arises
in the study of the long-time behavior of
finite-horizon MFGs. The stationary
(ergodic) MFG is determined by the following system of PDEs
\begin{equation}
	\label{eq:estacionario}
	\begin{cases}
		\De v+H(x, Dv)=\rho+\bF(x,m), & \\
		\De m-\diver(m\ D_pH(x, Dv))=0, & m>0, \,\,\int_{\T^d} m ~dx=1.
	\end{cases}
      \end{equation}
 It is usual to study the preceding system under periodic boundary conditions. Hence,
 the spatial variable $x$ takes values on the $d$-dimensional torus $\T^d$.
The data of this problem are a smooth Tonelli Hamiltonian, $H:\T^d\times\R^d\to\R$, 
   and a continuous coupling, $\bF: \T^d\times\cP(\T^d)\to\R$;
 the unknowns are  the distribution of players, $m\in \cP(\T^d)$, the value function, $v:\T^d\to\R$, and the ergodic constant, $\rho\in \R$.
Because of the periodicity, 
$v$ can only be determined up to additive constants. 

Much progress has been achieved in understanding stationary MFGs since the first results
in  \cite{ll1,ll3}
 on the 
 existence and uniqueness of weak solutions to
\eqref{eq:estacionario}.
For example, 
 in \cite{GM},
the existence of classical solutions for
\eqref{eq:estacionario} in the local coupling case
($\bF(\cdot,m)=F(m(\cdot))$)
was proven; these results were later improved in 
\cite{GPatVrt, PV15}.
The nonlocal case  with quadratic and sub-quadratic
Hamiltonian was addressed in \cite{BFl}. In parallel, the theory of Sobolev solutions
was developed in \cite{bocorsporr}. In 
\cite{FrS}, stationary MFGs with density constraints were addressed. Finally, 
general existence results using
monotone operator methods were developed in \cite{FG2}.

In contrast, the theory for the convergence of discrete schemes to \eqref{eq:estacionario} is less
well developed. The first systematic approaches to numerical methods for MFGs were discussed
in \cite{CDY} and \cite{AchdouCapuzzo10}.  See also the survey \cite{achdou2013finite}
and the subsequent paper \cite{AL16II}.
A new class of estimates for MFGs with congestion and the rigorous convergence of corresponding numerical schemes was discussed in 
\cite{Achdou2015, Achdou2016}. The ergodic problem was examined from a numerical perspective
in several works; for example,
\cite{cacace2018ergodic} studied the ergodic problem both analytically and numerically,  
 semi-discrete approximations were considered in \cite{MR2928379},
and semi-Lagrangian methods were investigated in 
\cite{MR3392626}. 
An alternative approach using proximal methods was developed in \cite{Briceno-Arias}. 
In the last section of this paper, we will use the ideas first developed in 
 \cite{almulla2017two} for justifying convergence of numerical schemes for monotone MFGs.


Let $L:\R^d\times\R^d\to\R$ represent the Legendre transform of the
Hamiltonian $H$,
 \[	L(x,v)=\sup_{p\in\R^d} pv-H(x,p).
\]

Here,
we introduce a discrete approximation \eqref{eq:disc-MFG} of the
system \eqref{eq:estacionario}.
This scheme is the MFG analog of the approximation 
scheme for Hamilton-Jacobi equations and Mather measures considered
in \cite{DIIS} and \cite{DIPPS}.

\begin{problem}\label{problem} Let $\tau>0$, $L:\T^d\times\R^d\to\R$
  be a smooth Tonelli Lagrangian, and
  $\bF:\T^d\times\cP(\T^d\times\R^d)\to\R$ be continuous. Find
  $(\rho,u,\tilde\mu)\in\R\times C(\T^d) \times\cC_\tau$ satisfying 
	\begin{equation}
		\label{eq:disc-MFG}
		\begin{cases}
			\tau\rho&=\cL_\tau u(x) -u(x)-\tau\bF(x,\mu), \\
			-\rho&=\displaystyle{\int_{\T^d\times\R^d} L(x,q)+\bF(x,\mu)\ d\tilde\mu(x,q)},
			
			\quad \mu={\Pr}_{1\#}\tilde\mu,
		\end{cases}
	\end{equation}
where the set of holonomic measures $\cC_\tau$ is introduced in
Definition \ref{def-C_tau}, and  the Lax operator
$\cL_\tau$ is defined by \eqref{def Lax}. 
\end{problem} 
Note that our discrete system depends on the parameter $\tau$. 
To stress that, we call it $\tau$-discrete MFG system.

Our primary goal is to establish the existence of solutions to
 the preceding problem  and, then, prove the
convergence, uniform in the nonlocal case and weak
in the local case, to classical solutions of
\eqref{eq:estacionario}. For that,  in 
Section \ref{backg}, 
 we review prior results on the approximation of stationary Hamilton--Jacobi equations and the  connection with Aubry--Mather theory. 
In Section \ref{sec:statement-problem}, we detail our discrete MFG problem and present its relation to other discrete models. 
Then, in Section \ref{sec:assumptions}, we present our main assumptions.

The existence of solutions is addressed in 
Section \ref{sec:solv-discr-mfg}. There, relying on Kakutani's fixed-point Theorem, we prove the following theorem. 
\begin{theorem}\label{goal1} Suppose that $L$ satisfies Assumption
  $(A1)$\footnote{see Section \ref{sec:assumptions} for details on the assumptions.}. 
  In the nonlocal case, further assume that  $\bF$ satisfies
  $(B1)$ and $(B2)$; in the local case,  
  suppose that  $\bF(\cdot,m)=F(m)$ satisfies
  $(B4)$ and $(B5)$. Then, there exists a solution $(\rho_\tau, u_\tau, \tilde\mu_\tau)$
	of the $\tau$--discrete MFG, \eqref{eq:disc-MFG}. In addition, $\rho_\tau$ and $\mu_\tau={\Pr}_{1\#}\tilde\mu_\tau$,  
	are unique and $u_\tau$ is unique up to additive constants.
	
	Moreover, in the local case when $F(m)=\log(m)$, we have
	\begin{align}\notag
	\rho_\tau&= \min_{\mu\in\cC_\tau}\int_{\T^d\times\R^d} L+\log\circ\mu\ d\tilde\mu\\
	\label{eq:nogao}
	&=
-\log\min_{\phi\in C(\T^d)} \int_{\T^d} e^{(\cL_\tau\phi(x)-\phi(x))/\tau} ~dx. 
	\end{align}
\end{theorem}

Next, in Section \ref{sec:aprox}, we establish our main results;
that is, we prove the convergence of the solutions of the $\tau$-discrete
MFG system, Problem \ref{problem}, to the classical solution of
\eqref{eq:estacionario}, when $\tau\to 0^+$. First, we prove helpful
properties of the operator  $\cL_\tau$ (see Subsection
\ref{sec:pre}). In Subsection \ref{sec:nonlocal}, we examine the
nonlocal coupling case and,  relying on semi-convexity estimates, prove

\begin{theorem}\label{nonlocal} Suppose that $L$ satisfies $(A1)$ and $(A2)$ and $\bF$ satisfies $(B1)$, $(B2)$, and $(B3)$. Let $x_0\in\T^d$ and
	let $(\rho_\tau,u_\tau,\tilde\mu_\tau)$ be a solution of the discrete
	MFG \eqref{eq:disc-MFG} with $u_\tau(x_0)=0$ and 
	$\mu_\tau={\Pr}_{1\#}\tilde\mu$. Let $(\rho,u,m)$ solves the
	ergodic MFG \eqref{eq:estacionario} with $u_0(x_0)=0$.  Then, 
	\begin{enumerate}[(i)]
		\item $\lim\limits_{\tau\to 0^+}\rho_\tau=\rho$,
                  \item $\lim\limits_{\tau\to 0^+}u_\tau=u$  uniformly,
		\item $\mu_\tau \rightharpoonup m$.
	\end{enumerate}
\end{theorem}
\begin{remark} Under the assumptions of Theorem \ref{nonlocal},
  Theorem 2.1 in \cite{BFl} ensures the existence of classical solutions to \eqref{eq:estacionario}. 
\end{remark}

The particular case of
\eqref{eq:disc-MFG} with the local coupling $F(m)=m^a$
is examined in Subsection \ref{bound.loc}. There,  
 we use the 
hypercontractivity of the Hamilton-Jacobi equation to get uniform
bounds for the solution of the $\tau$-discrete MFG
\eqref{eq:disc-MFG}. Using these bounds and using
the  monotonicity (see  \cite{Eva,FG2,FGT1}), we define weak solutions to
the $\tau$-discrete MFG system and to the ergodic MFG system (see
Definition \ref{weak}). Finally, in Section \ref{sec:weak-approach}, using Minty's method, 
we prove that
 normalized solutions to the discrete MFG system \eqref{eq:disc-MFG}
 weakly converge to a classical solution to
 \eqref{eq:estacionario}. 
 
 
\begin{theorem}\label{local}
	Suppose that $L$ satisfies $(A3)$ and $F$  satisfies $(B6)$.
	Let $(\rho_\tau,u_\tau,\tilde\mu_\tau)$ with  $\mu_\tau=\Pr_{1\#}\tilde\mu_\tau$ be the solution to the discrete
	MFG \eqref{eq:disc-MFG} such that $\max\eta^\tau*u_\tau=0$. Then, as $\tau\to 0$,
	$(\rho_\tau,\mu_\tau)$ converges in $\R\times\cP_\ell$ to
	$(r,m)$, $u_\tau$ converges weakly in $L^{1+1/a}$ to $u$, and
	$(r,u,m)$ is a classical solution to  \eqref{eq:estacionario}. 
\end{theorem}
\begin{remark} Under the assumptions of Theorem \ref{local},  Theorem 1 in \cite{PV15} ensures the existence of classical solutions to \eqref{eq:estacionario} for this local case. 
\end{remark}

\section{Background}
\label{backg}

Let $L:\R^d\times\R^d\to\R$ be a smooth Tonelli Lagrangian, $\Z^d$--periodic in the
space variable $x$, or equivalently, let $L$ be defined on $\T^d\times\R^d$.
Let $H:\R^d\times\R^d\to\R$, the Hamiltonian, be the Legendre transform of $L$.

Our discretized model is motivated by our previous study
in \cite{DIIS}
 of a
discrete approximation of the viscous
Hamilton--Jacobi equation,  
\begin{equation}\label{eq viscous HJ}
  \Delta u +H(x,Du)=\af_0,
\end{equation}
and by the discrete approximation of the stochastic Mather measures
introduced in \cite{G}, whose definition we recall next. 


Let $\cP(\T^d\times\R^d)$  be the set of Borel probability measures on $\T^d\times\R^d$.
Let $C^0_\ell$ be the set of continuous functions 
 $f: \T^d\times\R^d\to\R$ having at most linear  growth; that is,
 $$
 \lV f\rV_\ell:=\sup_{(x,q)\in \T^d\times\R^d}\frac{|f(x,q)|}{1+\lV q\rV}<+\infty.
 $$
 Let $\cP_\ell$  be the set of measures in $\cP(\T^d\times\R^d)$ such that
 $$
 \int_{\T^d\times\R^d}\lV q\rV\;d\mu<+\infty.
 $$
We endow $\cP_\ell$ with the topology  corresponding to the following convergence: for $\mu_n\in \cP_\ell$,  $\lim\limits_{n\to\infty}\mu_n=\mu$ if and only if
 \[
 \lim\limits_{n\to\infty}\int f\;d\mu_n=\int f\;d\mu,
 \]
 for all $f\in C^0_\ell$.
 
 Let $(C^0_\ell)'$ be the dual of $C^0_\ell$. Then,  $\cP_\ell$ is
 naturally embedded in $(C^0_\ell)'$, and its topology is 
the weak*- induced topology on $(C^0_\ell)'$. 
 This topology is metrizable.  To see this,
 let $\{f_n\}$ be a sequence of functions with compact support 
 on  $C^0_\ell$, which is dense on $C^0_\ell$ in the topology 
 of uniform convergence on compact sets of $\T^d\times\R^d$. 
 The metric $d$ on $\cP_\ell$ defined by
 \[
   d(\mu_1,\mu_2)=
   \Big|\int |q| \;d\mu_1-\int|q|\;d\mu_2\Big|
   +\sum_n\frac{1}{2^n\|f_n\|_\infty}
   \Big|\int f_n\; d\mu_1-\int f_n\; d\mu_2\Big|\]
 gives the topology of $\cP_\ell$.

 \begin{remark}\label{Ma}
According to \cite{Ma__1996}, for $c\in\R$, the set 
\[A(c):=\Big\{\tilde\mu\in \cP_\ell:\ \int_{\T^d\times\R^d} L\ d\tilde\mu\le c\Big\}\] 
is compact in $\cP_\ell$.
\end{remark}
\begin{definition}\label{def holonomic measure}
 We denote by  $\cC$ the set of {\em holonomic}
 measures; that is,  the  closed convex subset of measures in $\cP_\ell$ that satisfy
 \begin{equation*}
      \int_{\T^d\times\R^d} \left(\Delta\varphi(x)+\langle
    D\varphi(x),q\rangle \right)\ d\tilde\mu(x,q)=0
\end{equation*}
for all $\varphi\in C^2(\T^d).$
\end{definition}
The set $\cC$ is non-empty as the  measure $\tilde\nu$ given by
  \begin{equation}
    \label{eq:example}
\int f d\tilde\nu = \int f(x,0) dx,\quad f\in C^0_\ell 
  \end{equation}
is holonomic.
  
  More generally, let $V:\T^d\to\R^d$ be a $C^1$ vector field and let
$\mu\in \cP(\T^d)$ be the unique weak solution of the following
Fokker--Planck equation.
\begin{equation}\label{eq FK}
\Delta\mu-\diver\left(V(x)\mu\right)=0\qquad\hbox{in $\T^d$}.
\end{equation}
For $G_V(x)=(x,V(x))$, we have that $\tilde\mu:=G_{V\#}\mu$ is a
holonomic measure.

According to \cite{G}, the ergodic constant in \eqref{eq
  viscous HJ} is determined both by the minimization problem
\begin{equation*}
  \af_0=\min_{\fui\in C^2(\T^d)}\max_x \Delta\fui(x) +H(x,D\fui(x))
\end{equation*}
and by the dual problem
\begin{equation}\label{eq:dual}
  -\af_0:=\min_{\tilde\mu\in\cC} \int_{\T^d\times\R^d} L\ d\tilde\mu.
\end{equation}
Holonomic measures achieving the minimum in \eqref{eq:dual} 
are called {\em stochastic Mather measures}.

\begin{theorem}[\cite{G}]\label{teo stochastic Mather problem}
A measure $\tilde\mu$ is a stochastic Mather measure if and only if 
$\tilde\mu:={G_V}_{\#}\mu$ where $\mu\in\cP(\T^d)$ is the solution of \eqref{eq FK} 
with $V(x):=D_p H(x,Du(x))$ and $u$ is any solution to
\eqref{eq viscous HJ}
\end{theorem}
Because solutions to \eqref{eq viscous HJ} are unique, up to additive constants,
there is only one stochastic Mather measure for each Lagrangian.

\section{The discrete MFG problem}
\label{sec:discrete-mfg-problem}

Now, we examine in detail Problem \ref{problem} and explore related discrete MFG models. 

\subsection{Statement of the problem}
\label{sec:statement-problem}

For the discrete problem, the source of randomness corresponds to the heat kernel
$\eta^\tau$  on $\T^d=\R^d/\Z^d$; that is, 
\begin{align*}
  \eta^\tau(z)&=\frac{1}{(4\pi\tau)^{\frac d2}}
\sum_{k\in\Z^d}e^{-\frac{|z+k|^2}{4\tau}}
=\sum_{k\in\Z^d}e^{-\tau|2\pi k|^2+2\pi ik\cdot z}.
  \end{align*}
  Accordingly, for $u\in C(\R^d)$, $\Z^d$--periodic, we have
  \[ (\eta^\tau*u)(y)=\int_{\T^d}  \eta^\tau(y-z) u(z)\ dz
=\frac{1}{(4\pi\tau)^{\frac d2}}\int_{\R^d}e^{-\frac{|y-z|^2}{4\tau}}u(z)\ dz. \]
Next, we define the operator $\cL_{\tau}:C(\T^d)\to C(\T^d)$ as  
\begin{equation}\label{def Lax}
  \cL_\tau u(x):=\max_{q\in\R^d}\big((\eta^\tau*u)(x+\tau q)-\tau L(x, q)\big)
  \hbox{ for  }x\in \R^d.  
\end{equation}

As discussed in \cite{DIIS},
the 
discrete version \eqref{eq viscous HJ}, 
\begin{equation}\label{discreteHJ}
	u=\cL_\tau u-\tau\af_\tau,  
\end{equation}
arises from 
the discretization of the stochastic Lax formula.

We recall the following results concerning \eqref{discreteHJ} that were 
established in \cite{DIIS}.
\begin{enumerate}[(i)]
\item   There is a unique value $\af_\tau\in\R$ for which \eqref{discreteHJ}
  admits solutions. This value is given by 
  \begin{equation}\label{eq:constant}
    \af_\tau=
    \min_{\fui\in C(\T^d)}\max_{x,q}\frac{(\eta^\tau*\fui)(x+\tau q)-\fui(x)}\tau- L(x,q)  
  \end{equation}
  and satisfies
  \[\min_{\T^d\times\R^d} L\le -\af_\tau\le \max_{x\in\T^d}\min_{q\in\R^d}L(x,q).\]
\item A solution $u\in C(\T^d)$ to \eqref{discreteHJ} is unique up to additive constants.
\end{enumerate}

For the continuous second-order Hamilton-Jacobi equation, the preceding results 
were established in \cite{G}.

Now, we fix $u\in C(\T^d)$ and  let
\[
V_u(x):=\arg\max[(\eta^\tau * u)(x+\tau q) -\tau L(x,q)].
\]
Then, we have the following two additional results.

  \begin{enumerate}[(i)]\setcounter{enumi}{2}
   \item If $q\in \V_u(x)$ then \[D(\eta^\tau*u)(x+\tau q)=D_qL(x,q)\]
  or, equivalently,
  \[q= D_pH(x ,D(\eta^\tau*u)(x+\tau q)).
  \]

\item 
There is a Borel measurable map $V:\T^d\to\R^d$ that is optimal for
$u$; that is, $V(x)\in \V_u(x)$ for all $x\in\T^d$. Moreover, there is $\tau_u>0$
such that  for all $\tau<\tau_u$, $\V_u(x)$ is a singleton and
\[D^2_{qq}\left[\frac{\eta^\tau*u(x+\tau q)-u(x)}\tau- L(x,q)\right]=
\tau D^2(\eta^\tau*u)(x+\tau q)-D^2_{qq}L(x,q)\]
is negative definite for $q=\V_u(x)$.

\item If $u\in C(\T^d)$ and $V:\T^d\to\R^d$ Borel measurable satisfies
  \begin{equation}\label{eq:sol-control}
     (\eta^\tau*u)(x+\tau V(x))-u(x)-\tau L(x,V(x))=\tau\af_\tau, 
  \end{equation}
then $u$ solves \eqref{discreteHJ} and $V$ is optimal for
$u$. 
\end{enumerate}

We follow the generalized scheme in \cite{MR2458239} to describe the dual of the
minimization problem \eqref{eq:constant}.

\begin{definition}\label{def-C_tau}
We denote by $\cC_\tau$ the set of $\tau$--holonomic
measures; that is, the closed convex subset of measures
$\tilde\mu\in\cP_\ell$ such that the projected measure
 $\mu=\Pr_{1\#}\tilde\mu$ satisfies
\[\mu(A)=\int_A \int_{\T^d\times \R^d} \eta^\tau(x+\tau q-z)
  \ d\tilde\mu(x,q)  \ dz \]
for any Borel set $A\subset\T^d$.
\end{definition}
We observe that
\[m_{\tilde\mu}(z)=\int_{\T^d\times \R^d} \eta^\tau(x+\tau q-z)
  \ d\tilde\mu(x,q)\]
defines a smooth function. Thus, when $\tilde\mu$ is $\tau$--holonomic,
its projected measure has a smooth density that we identify
with this projected measure. Note that we have the following upper bound
 \begin{equation}
 \label{ubm}
 m_{\tilde\mu}\le (4\pi\tau)^{-d/2}.
\end{equation}

The dual problem to \eqref{eq:constant} is
\begin{equation}
  \label{discrete-dual}
  -\af_\tau=\min\left\{\int_{\T^d\times\R^d} L\ d\tilde\mu:\tilde\mu\in
    \cC_\tau\right\}.
  \end{equation}

The measure given in \eqref{eq:example} is $\tau$-holonomic. 
More generally,

\begin{proposition}\label{prop tau-holonomic}\cite{DIPPS}
  Any Borel measurable map $V:\T^d\to\R^d$, defines a
    $\tau$--holonomic measure $\tilde\mu^V$ such that
  \begin{enumerate}[(a)]
  \item If $u\in C(\T^d)$ and $V$ satisfy  \eqref{eq:sol-control}, then
  \[-\af_\tau=\int_{\T^d\times\R^d} L\; d\tilde\mu^V.\]
  \item   If $V$ is continuous there
   is $\mu\in\cP(\T^d)$ such that $\tilde\mu^V=G_{V\#}\mu$, and in
   particular, for any $f\in C(\T^d)$ we have 
\[\int_{\T^d} \eta^\tau*f(x+\tau V(x))-f(x) \;d\mu(x) =0.\]
 \end{enumerate}
\end{proposition}
The previous discussion motivates the $\tau$-discrete MFG model
\eqref{eq:disc-MFG}, for $(\rho,u,\tilde\mu)\in\R\times C(\T^d) \times\cC_\tau$,
\begin{equation*}
	\begin{cases}
		\tau\rho&=\cL_\tau u(x) -u(x)-\tau\bF(x,\mu) \\
		-\rho&=\displaystyle{\int_{\T^d\times\R^d} L(x,q)+\bF(x,\mu)\ d\tilde\mu(x,q)},
		\quad \mu={\Pr}_{1\#}\tilde\mu,
	\end{cases}
\end{equation*}
that was presented in Problem \ref{problem}. 

\subsection{Relation with other discrete MFG models}
\label{sec:comparison}

The work \cite{Saldi2019DiscretetimeAM} examined a model for discrete average cost MFG. 
The model is determined by a quadruple, $(X,A,p,c)$, where $X$ and $A$ stand for the state and control spaces, $p$ is a stochastic kernel, $p:X\times A\times\cP(X)\to\cP(X)$, that gives the transition probability law, and  $c:X\times A\times\cP(X)\to[0,\infty)$ is the one-stage cost. A policy is a stochastic kernel on $A$ given $X$, $\pi:X\to\cP(A)$.
Each policy, together with an initial distribution $\mu_0\in\cP(X)$ and
the transition probability $p$, defines a measure $P^\pi$ on $(X\times A)^\infty$.
We denote the expectation with respect to $P^\pi$ by $\E^\pi$.

Given a policy $\pi$, the corresponding average cost is 
\[ J_\mu(\pi)
  =\limsup_{n\to\infty}\frac 1n\E^\pi\left[\sum_{i=0}^{n-1}c(x_i,a_i,\mu)\right].\]
The average cost problem consists of finding $\pi^*$ that is optimal for
$\mu$; that is, such that
\[J_\mu(\pi^*)=\inf_\pi J_\mu(\pi).\]

For any set  $A$ let $2^{A}$ be the set of all subsets of $A$. 
Consider the map $\Psi:\cP(X)\to 2^\Pi$ given by 
\[\Psi(\mu)=\{\pi\in\Pi:\pi \hbox{ is optimal for }
\mu\}.\]
Furthermore, define the map $\Lam:\Pi\to 2^{\cP(X)}$ as follows: given
$\pi\in\Pi$, $\mu\in\Lam(\pi)$ if
\[\mu(B)=\int_{X\times A}p(B\mid x,a,\mu)\pi(da\mid x)\mu(dx).
\]

Under the assumption
discussed in \cite{Saldi2019DiscretetimeAM},
$\Lam(\pi)$ has a unique element for all $\pi$.
A pair $(\pi,\mu)\in\Pi\times\cP(X)$ is a {\em mean-field equilibrium} 
if $\pi\in\Psi(\mu)$ and $\mu\in\Lam(\pi)$.

For us, the state and control spaces are the torus $\T^d$ and $\R^d$.
The transition law $p:\T^d\times\R^d\to\cP(\T^d)$ is given by the dynamics 
\begin{equation*}
x_{i+1}=x_i+\tau q_i+\nu_i,  
\end{equation*}
where $(\nu_i)$ is a sequence of Gaussian random variables, i.i.d. with common law
\[\eta^\tau(s)=\frac{1}{(4\pi\tau)^{\frac d2}}e^{-|s|^2/4\tau}.\]
 Thus,
\[p(B\mid x,q)=\int_{\T^d}\chi_B(x+\tau q+s) \eta^\tau(s)ds
  =\int_{B}\eta^\tau(z-x-\tau q)dz.\]
The one-stage cost function is $c_\tau(x,q,\mu)=\tau (L(x,q)+F(\mu))$.

To obtain a mean-field equilibrium $(\pi,\mu)$ from a solution of
the MFG system \eqref{eq:disc-MFG}, we disintegrate $\tilde\mu$ as
$\tilde\mu(dx,dq)=\pi(dq\mid x)\mu(dx)$. This approach was used in \cite{Saldi2019DiscretetimeAM}. The setting in that reference is substantially different from ours: there, the control space $A$ is compact and the one-stage cost $c$ is bounded.
In particular, the results in \cite{Saldi2019DiscretetimeAM} do not imply our results
directly.

\section{Assumptions}
\label{sec:assumptions}
Here, we describe the main assumptions used in the sequel. 
\begin{enumerate}
\item The Lagrangian $L:\T^d\times\R^d\to\R$ is  smooth. Moreover, 
we will need some of the following assumptions for our results. 
\begin{enumerate} 
	\item[(A1)]
 There is $c_0>0$ such that $D^2_{qq}L\ge c_oI$.
 	\item[(A2)]   
There are $k_1, k_2$, such that for any $x, y, q, v\in\R^d$
\[L(x+h,q+v)-2L(x,q)+L(x-h,q-v)\le k_1|h|^2+k_2|v|^2.\]
	\item[(A3)]
The Lagrangian  $L$ is separable; that is, $L(x,v)=K(v)-U(x)$. Furthermore,  $K(v)\ge c|v|^q$, 
$q=2$ for $d=1$ and $q>d$ for $d>2$; $K(0)=0$.
\end{enumerate}
Assumption (A1) is used for the existence of solutions. Notice that Assumption (A1) implies that there are $c_1, c_2>0$ such that  for any $x\in\T^d$, $q\in\R^d$
\begin{equation}\label{eq:super}
	|q|^2\le c_1L(x,q)+c_2  
\end{equation}

For the convergence, in the nonlocal case, we work under Assumption (A2); for the local case, instead, we use (A3).  Observe that for $d=1$, a Lagrangian that satisfies Assumption
(A3) also satisfies Assumption (A2).

 \item The coupling $\bF:\T^d\times\cP(\T^d)\to\R$ is either nonlocal or local. 
For our various results, we will work with the following assumptions.
    \begin{enumerate}
    \item {\bf Nonlocal.} 
    We assume that
\begin{enumerate}
	\item[(B1)] $\bF$ is continuous.
	
		\item[(B2)]    $\bF$ stricly monotone; that is,
	
	\[\int_{\T^d}(\bF(x,m)- \bF(x,m'))d(m-m')(x)> 0\quad\hbox{for any }m\ne m'.\]
	
	\item[(B3)] There is $k_0>0$ such that for  any $x,h\in\R^d$, $m\in\cP(\T^d)$
	\[\bF(x+h,m) -2\bF(x,m)+\bF(x-h,m)\le k_0|h|^2.\]   

\end{enumerate}
The monotonicity assumption is used to establish uniqueness, as it is standard in MFG theory. The last assumption provides uniform semiconcavity bounds that are crucial for the convergence results.

 \item {\bf Local.} We assume $\bF(x,m)=F(m(x))$ with $F:\R^+\to\R$ smooth and, 
 we will work with the following additional properties. 
\begin{enumerate}
	\item[(B4)]                       
 $F$ concave and strictly increasing
	\item[(B5)] 
 $F$ bounded  below or $F(m)=\log m$.
  	\item[(B6)]  $F$ has the following form  $F(m)= m^\alpha$, $0<\alpha<1$
\end{enumerate}  
Assumptions (B6) is used in the analysis of the convergence. Of course, (B6) implies (B5)
since $m^\alpha$ is bounded by below. 
 \end{enumerate}
\end{enumerate}

In the sequel, we will use the following constants
  \begin{enumerate} 
  \item In the nonlocal case, $a_\bF=\min\bF$, $b_\bF=\max\bF$.
  \item In the local case,  $a_\bF=\min F$  when $F$ is bounded below,
    $a_\bF=0$ when $F(m)=\log m$, and $b_\bF=F(1)$.
  \end{enumerate}

\section{Solving the discrete MFG}
\label{sec:solv-discr-mfg}

Fix $\tilde\mu\in\cC_\tau$ and let $\mu={\Pr}_{1\#}\tilde\mu$. 
According to properties  (i)
and (ii) concerning \eqref{discreteHJ} that were 
discussed in subsection \ref{sec:statement-problem},
there is a unique  $\rho^\mu\in\R$
and $u_\mu\in C(\T^d)$, unique up to additive constants, solving the first
equation of \eqref{eq:disc-MFG}.
Thus, it follows from \cite{DIPPS} that the set
\begin{equation}\label{eq:min-map}
    \Psi_\tau(\tilde\mu):=  \Bigl\{\tilde\nu\in 
  \cC_\tau: \int_{\T^d\times\R^d} L(x,q)+\bF(x,\mu)\; d\tilde\nu(x,q) =-\rho^\mu\Bigr\}
  \end{equation}
  is non-empty. Furthermore, $\Psi_\tau(\tilde\mu)$ is also convex.

  \begin{lemma}\label{unique} Assume $L$ satisfies (A1) and $\bF$
    satisfies (B2) or (B4).
    Let
  $(\rho^1, u^1, \tilde\mu^1)$, $(\rho^2, u^2, \tilde\mu^2)$ solve the
  $\tau$--discrete MFG in \eqref{eq:disc-MFG}. Then $\Pr_{1\#}\tilde\mu^1=\Pr_{1\#}\tilde\mu^2$, $\rho^1=\rho^2$
  and $u^1-u^2$ is constant.
\end{lemma}
\begin{proof} Let $\mu^i=\Pr_{1\#}\tilde\mu^i$, $i=1,2$. Then,
  \begin{align*}
      \int_{\T^d\times\R^d} L(x,q)+\bF(x,\mu^1)\ d\tilde\mu^1(x,q)&=-\rho^1
    \le  \int_{\T^d\times\R^d} L(x,q)+\bF(x,\mu^1)\ d\tilde\mu^2(x,q),
      \end{align*}
  and
    \begin{align*}
    \int_{\T^d\times\R^d} L(x,q)+\bF(x,\mu^2)\ d\tilde\mu^2(x,q)&=-\rho^2
    \le \int_{\T^d\times\R^d} L(x,q)+\bF(x,\mu^2)\ d\tilde\mu^1(x,q).
  \end{align*}
Adding these inequalities, we have
\begin{multline*}
  \int L d\tilde\mu^1+\int L d\tilde\mu^2+\int \bF(x,\mu^1 )d\mu^1(x)+
  \int\bF(x,\mu^2)d\mu^2(x)\\\le \int Ld\tilde\mu^1+\int L d\tilde\mu^2+
  \int\bF(x,\mu^1 )d\mu^2(x)+\int \bF(x,\mu^2)d\mu^1(x).
  \end{multline*}
Therefore,
\[\int_{\T^d} (\bF(x,\mu^1)-\bF(x,\mu^2))d(\mu^1-\mu^2)(x)\le 0.\]
By the strict monotonicity of $\bF$, in the nonlocal case, or the strict monotonicity of $F$, in
the local case, if we had $\mu^1\ne\mu^2$, then we would have
\[\int_{\T^d} (\bF(x,\mu^1)-\bF(x,\mu^2))d(\mu^1-\mu^2)(x)> 0.\]
Thus, $\mu^1=\mu^2$, so $\rho^1=\rho^2$. Furthermore, by
  the fact (ii) about  equation \eqref{discreteHJ}, 
  we  have that $u^1-u^2$ is constant.
\end{proof}
 To solve the discrete MFG, we must show that the map
  $\Psi_\tau:\cC_\tau\to 2^{\cC_\tau}$ has a fixed point. This is achieved 
   using  Kakutani's fixed-point theorem.

  \begin{proposition}\label{bound} 
 	Suppose that $L$ satisfies (A1) and 
   that either $\bF$ is nonlocal   satisfying $(B1)$ or
   that $\bF$  is 
    local with $F$ satisfying $(B4)$ and $(B5)$. Moreover, if 
 $\bF$ is local and $F(m)=\log m$ suppose that  $0<\tau<2/c_1$.

  Then, there is a constant, $C_\tau>0$, such that for
\[A_\tau=\Big\{\tilde\mu\in\cC_\tau:\int_{\T^d\times\R^d} L\; d\tilde\mu\le C_\tau\Big\},\]
  $\bF|_{\T^d\times A_\tau}$ is bounded and
  $\Psi_\tau(A_\tau)\subset 2^{A_\tau}$.  
  \end{proposition}

  \begin{proof}
    Consider the problem in \eqref{discrete-dual}.
  In the nonlocal case, $(B1)$ and the compactness of
       $\T^d\times\cP(\T^d)$ imply that $\bF$ is bounded.
       In the local case, from $(B4)$
       and the upper bound \eqref{ubm}, we have $ F(\mu)\le
       F((4\pi\tau)^{-d/2})  $, for $\mu=\Pr_{1\#}\tilde\mu $, $\tilde\mu\in\cC_\tau$.

We examine first the following case: $\bF$ is nonlocal   satisfying $(B1)$ or
$\bF$  is local with $F$ satisfying $(B4)$ and bounded by
below. We address the logarithmic case separately.

By the definition of $\Psi_\tau(\tilde\mu)$ in \eqref{eq:min-map}, and applying Proposition  \ref{prop tau-holonomic} using  the measure defined in \eqref{eq:example}, 
we get
\begin{equation}\label{pro2-rho-bound}
-\rho^\mu\le \int_{\T^d} L(x,0)+\bF(x,\mu)\; dx \le
\max_{x\in\T^d}L(x,0)+b_\bF,
\end{equation}
  	where, in the local case, we use the concavity of $\bF$. Hence, in the case $(1)$ for $\tilde\nu\in\Psi_\tau(\tilde\mu)$, we have 
    \[\int_{\T^d\times\R^d} L\; d\tilde\nu+a_\bF\le -\rho^\mu\le \max_{x\in\T^d}
    \left[L(x,0)+b_{\bF}-\bF\right].\]
    Thus, it suffices to take
    $C_\tau=\max\limits_{x\in\T^d}L(x,0)+b_\bF-a_\bF$.

    Now, we consider the case where $F(m)=\log m$. From  the convexity of the exponential
    and the fact that $\mu \in \cC_\tau$, we have
    \begin{align*}
    \mu(z)&=(4\pi\tau)^{-\frac d2}
\sum_{k\in\Z^d }\int_{\T^d\times\R^d} \exp\left( -\frac{|x+\tau q-z+k|^2}{4\tau}\right) \ d\tilde\mu(x,q)\\
&\ge(4\pi\tau)^{-\frac d2}\sum_{k\in\Z^d}\exp\left( \int_{\T^d\times\R^d}-\frac{|x+\tau q-z+k|^2}{4\tau}\ d\tilde\mu(x,q)\right) \\
   &\ge(4\pi\tau)^{-\frac d2}\sum_{k\in\Z^d}
 \exp\left( \int_{\T^d\times\R^d}-\frac{(|k|+1)^2}{2\tau}-\frac\tau 2|q|^2\ d\tilde\mu(x,q)\right) \\
&= \exp\left(-\int_{\T^d\times\R^d}\frac\tau 2|q|^2\ d\tilde\mu(x,q)\right)B_\tau, 
\end{align*}
where
\begin{equation*}
B_\tau:=(4\pi\tau)^{-\frac d2}\sum\limits_{k\in\Z^d}e^{-\frac{(|k|+1)^2}{2\tau}}. 
\end{equation*}
From  \eqref{eq:super}, we have
    \begin{align} 
 \log\mu(z)
\ge&-\frac \tau 2
    \left(\int_{\T^d\times\R^d} c_1 L\; d\tilde\mu+ c_2\right) +\log B_\tau.
\label{Fbound}
    \end{align}
Therefore, by \eqref{pro2-rho-bound} for $\tilde\nu\in\Psi_\tau(\tilde\mu)$, we get
\begin{align*}
\int_{\T^d\times\R^d} L\; d\tilde\nu+\log B_\tau-\frac \tau 2
  \left(\int_{\T^d\times\R^d} c_1 L\; d\tilde\mu+ c_2\right)\le -\rho^\mu
  &\le \max_{x\in\T^d}L(x,0).
\end{align*}
Consequently,
\[\int_{\T^d\times\R^d} L\; d\tilde\nu\le \max_{x\in\T^d}L(x,0)-\log B_\tau
  +\frac\tau 2\left(c_1\int_{\T^d\times\R^d} L\; d\tilde\mu+c_2\right).\]
  For $\tau<2/c_1$, let
  \[C_\tau=(\max_{x\in\T^d}L(x,0)-\log B_\tau+\frac\tau 2c_2)/(1-\frac
    \tau 2 c_1).\]
It follows from \eqref{Fbound} that $\bF|_{\T^d\times A_\tau}$ is bounded from below.
  If  $\int\limits_{\T^d\times\R^d}L\; d\tilde\mu\le C_\tau$, then
  \[\int_{\T^d\times\R^d} L\; d\tilde\nu\le C_\tau(1-\frac\tau 2 c_1)
    +\frac\tau 2 c_1 C_\tau=C_\tau\]
  and so $\Psi_\tau(A_\tau)\subset 2^{A_\tau}$.
\end{proof}
\begin{lemma}\label{fix-point} Assume the hypothesis of Proposition
  \ref{bound}. Then,  the map $\Psi_\tau$ has a fixed point.
\end{lemma}
\begin{proof}
  Let $C_\tau>0$ and $A_\tau\subset\cC_\tau$ be given by Proposition
  \ref{bound}. Consider the map
  $\Psi_\tau:A_\tau\to 2^{A_\tau}$.
The set  $A_\tau$ is convex. Moreover, according to Remark \ref{Ma}, $A_\tau$ is also compact. Thus, to apply Kakutani's theorem, it suffices to check that $\Psi_\tau$ has a closed graph.
  For this, let $\tilde\mu_n, \tilde\nu_n$ be sequences in $A_\tau$
  converging to $\tilde\mu$ and $\tilde\nu$
  respectively such that $\tilde\nu_n\in\Psi(\tilde\mu_n)$.
  Hence, for any $\tilde\eta\in\cC_\tau$, we have
  \[\int_{\T^d\times\R^d} L(x,q)+\bF(x, \mu_n)\ d\tilde\nu_n(x,q)
    \le\int_{\T^d\times\R^d} L(x,q)+\bF(x, \mu_n)\ d\tilde\eta(x,q).\]
    Note that the sequences of smooth functions $\mu_n=\Pr_{1\#}\tilde\mu_n$,
    $\nu_n=\Pr_{1\#}\tilde\nu_n$ converge pointwise to $\mu=\Pr_{1\#}\tilde\mu$ and
    $\nu=\Pr_{1\#}\tilde\nu$, respectively.

  Since $\bF|_{\T^d\times A_\tau}$ is bounded and
    $0 \le\nu_n\le(4\pi\tau)^{-\frac d2}$,  by dominated convergence,
    \[ \lim \int _{\T^d}\bF(x,\mu_n) d\nu_n(x)=\int _{\T^d}\bF(x,\mu) d\nu(x),\quad
    \lim\int_{\T^d}\bF(x,\mu_n) d\eta=
    \int_{\T^d}\bF(x,\mu) d\eta.\]
 Therefore, recalling that $L$ is bounded from below, we deduce
  \[   \int_{\T^d\times\R^d} L+\bF(x,\mu)\ d\tilde\nu
    \le\liminf\int_{\T^d\times\R^d} L+\bF(x, \mu_n)\ d\tilde\nu_n
    \le\int_{\T^d\times\R^d} L+\bF(x,\mu)\ d\tilde\eta,\]
showing that $\tilde\nu\in\tilde\Psi_\tau(\mu)$. Thus, $\Psi_\tau$ has a closed graph. Hence, by 
 Kakutani fixed-point Theorem, it follows  that there is
$\tilde\mu\in A_\tau$ such that $\tilde\mu\in\Psi(\tilde\mu)$.
\end{proof}
Now, we are ready to prove Theorem \ref{goal1}.

\begin{proof} [Proof of Theorem~\ref{goal1}]  The existence and uniqueness of solutions to the $\tau$--discrete MFG, \eqref{eq:disc-MFG},  follows  from Lemmas
	\ref{unique} and \ref{fix-point}. To complete the proof, it remains to prove \eqref{eq:nogao} for the local case when $F(m)=\log(m)$.
First,  we claim that for any $\tilde{\mu}\in\cC_\tau$ with $\mu={\Pr}_{1\#}\tilde\mu$  and $\phi\in C(\T^d)$, the following inequality holds
\begin{equation}
\label{upbnd}
 \int_{\T^d\times\R^d} L+\log(\mu)\ d\tilde\mu \ge  -\log \int_{\T^d} e^{(\cL_\tau\phi(x)-\phi(x))/\tau}dx.
 \end{equation}
To verify \eqref{upbnd}, let $\tilde\mu\in\cC_\tau$ with $\mu={\Pr}_{1\#}\tilde\mu$ and $\phi\in C(\T^d)$. We have
\begin{align*}
  \int L(x,v) d\tilde\mu &=
\int \left( L(x,q) -\frac{\phi(x) - \eta^\tau*\phi(x+\tau q)}{\tau}\right) d\mu\\ 
&\ge-\int\frac{\phi(x) - \cL_\tau\phi(x)}{\tau}d\mu.
\end{align*}
Define
\[
a(m,\phi)=\int \left(\frac{\cL_\tau\phi(x)-\phi(x)}{\tau}-\log (m)\right)\ dm,
\]
for any absolutely continuous Borel probability measure $m$ in $\T^d$.
Then
\[\int_{\T^d\times\R^d} L+\log(\mu)\ d\tilde\mu \ge -a(\mu,\phi).\] 

Let
\[\lambda_\phi=\log \int e ^{(\cL_\tau\phi(x)-\phi(x))/\tau}dx, \]
and
\begin{equation}
\label{mphi}
m_\phi(x)=e^{(\cL_\tau\phi(x)-\phi(x))/\tau-\lambda_\phi}.
\end{equation}
Note that $a(m_\phi,\phi)=\lam_\phi$. Recall that 
 $t-t\log t\le 1$ for all $t>0$. 
 Therefore, 
  for any probability measure $m$ in $\T^d$, absolutely continuous with respect to
the Lebesgue, we have 
\begin{equation*}
m\log(m)-m_{\phi}\log(m_{\phi})\geq (1+\log(m_{\phi}))(m-m_{\phi}).
\end{equation*}
Consequently, we obtain
\[a(m,\phi)\le a(m_\phi,\phi)\\
+\int\Bigl(\frac{\cL_\tau\phi(x)-\phi(x)}{\tau}-\log (m_\phi)-1)(m(x)-m_\phi(x)\Bigr) dx.
\]
Recalling that  $m$ and $m_\phi$ are
probability measures  \eqref{mphi}  implies that the second term on the right-hand side vanishes. Thus,
\[\lam_\phi=\sup_m a(m,\phi).\]
This implies \eqref{upbnd}. Therefore,
\begin{equation}
  \label{eq:desigual} 
 \inf_{\mu\in\cC_\tau}\int\limits_{\T^d\times\R^d} L+\log(\mu)\ d\tilde\mu\ge
-\log\inf_{\phi\in C(\T^d)} \int_{\T^d} e^{(\cL_\tau\phi(x)-\phi(x))/\tau}dx.
\end{equation}

  From the first equation in \eqref{eq:disc-MFG}, for  $\mu_\tau={\Pr}_{1\#}\tilde\mu_\tau$, we have
\[ \exp\Bigl(\frac{\cL_\tau u_\tau(x)-u_\tau(x)}{\tau}\Bigr)=e^{\rho_\tau}\mu_\tau.\]
Thus, from \eqref{eq:desigual}, we get
\begin{align*}
\inf_{\mu\in\cC_\tau}\int\limits_{\T^d\times\R^d} L+\log(\mu)\ d\tilde\mu&\ge
-\log\inf_{\phi\in C(\T^d)} \int_{\T^d}e^{(\cL_\tau\phi(x)-\phi(x))/\tau}\ge
-\log \int_{\T^d} e^{(\cL_\tau u_\tau(x)-u_\tau(x))/\tau}\\ &
=-\log\int_{\T^d}e^{\rho_\tau} d\mu_\tau=-\rho_\tau=
\int\limits_{\T^d\times\R^d} L+\log(\mu_\tau)\ d\tilde\mu_\tau   \qedhere
\end{align*}
\end{proof}
\begin{remark}\label{support}
Let  $(\rho, u, \tilde\mu)$ with $\mu=\pi_\#\tilde\mu$ solve the $\tau$-discrete MFG, \eqref{eq:disc-MFG}, and   let
  $\V_u(x):=\arg\max[(\eta^\tau * u)(x+\tau q) -\tau L(x,q)]$. Since
\[ 0=\int (\eta^\tau * u)(x+\tau q)-u(x)\ d\tilde\mu\le
  \int\tau (L(x,q)+F(\mu(x))+\rho)\ d\tilde\mu=0, \]
we have that $\tilde\mu$ is supported in the set $\{(x,q):q\in\V_u(x) \}$.  
\end{remark}
  \begin{remark}\label{rel-comp}
For a solution $(\rho_\tau, u_\tau, \tilde\mu_\tau)$ of the
$\tau$--discrete MFG, \eqref{eq:disc-MFG},  we have 
\begin{equation}
  \label{eq:const-bound}
  \min_{\T^d\times\R^d} L+a_\bF\le -\rho_\tau\le \max_{x\in\T^d}L(x,0)+b_\bF
\end{equation}
  and
  \begin{equation*}
  \min_{\T^d\times\R^d} L\le \int\limits_{\T^d\times\R^d} L\ d\tilde\mu\le \max_{x\in\T^d}L(x,0)+b_\bF-a_\bF .   
  \end{equation*}

  Thus, by Remark \ref{Ma}, we have that
  \[\{(\rho_\tau, \tilde\mu_\tau): (\rho_\tau, u_\tau,
    \tilde\mu_\tau)\hbox{ solves \eqref{eq:disc-MFG}}\} \]
is precompact in $\R\times\cP_\ell$.
\end{remark}

\section{Convergence}
\label{sec:aprox}
Now, we study the convergence of 
solutions of the discrete MFG problem, \eqref{eq:disc-MFG} to 
solutions to the MFG system, \eqref{eq:estacionario}. We begin with proving some preliminary results, which are needed to prove Theorems \ref{nonlocal} and \ref{local}. We treat  nonlocal and local cases separately. First, in Subsection \ref{sec:nonlocal}, we examine the nonlocal case. We begin by proving the uniformly semi-convexity of solutions of the discrete  problem.  This estimate leads to the proof of Theorem \ref{nonlocal}. In Subsection \ref{bound.loc}, 
we address the local case. There, 
using the hypercontractivity of the Hamilton-Jacobi equation, we establish uniform bounds for the solutions of the discrete MFG problem in the local case. Finally, relying on these bounds and using Minty's method, in Subsection \ref{sec:weak-approach}, we obtain Theorem \ref{local}.

\subsection{Preliminary Computations}
\label{sec:pre}
In this section, we prove several results, which are needed to get convergence of the solutions of discrete MFG \eqref{eq:disc-MFG} to solutions of  the ergodic  problem \eqref{eq:estacionario}.

We begin by recalling the following Lemma from \cite{DIIS}.
\begin{lemma}\label{discrete-continuous}\cite{DIIS}
	Let $\fui\in C^{2}(\T^d)$. For every $R>0$, there exists a non-decreasing
	continuous function $\om:[0,+\infty)\to [0,+\infty)$ vanishing at $0$, only depending on $R$ and $\fui$, such that 
	\begin{equation*}
	\left|\frac{(\eta^\tau*\fui)(x+\tau q) - \fui(x)}{\tau}-\langle D\fui(x),q\rangle -\Delta\fui(x)\right|
\leq \om(\tau),
	\end{equation*}
	for all $(x,q)\in \T^d\times B_R$ and $\tau>0$.
\end{lemma}

\begin{proposition}\label{implicit}
  Assume that $L$ satisfies $(A1)$. Let $\fui\in C^2(\T^d)$. Then, there exist $t_\fui>0$ and $V_\fui\in C_1^2([0,t_\fui)\times\T^d,\R^d)$ such that $V_\fui(x,0)=D_pH(x,D\fui(x))$ and
  \[\cL_\tau\fui(x)=(\eta^\tau*\fui)(x+\tau
	V_\fui(\tau,x))-\tau L(x,V_\fui(\tau,x)) \hbox{ for } \tau<t_\fui, x\in\T^d.\]
\end{proposition}
\begin{proof}
  We aim to prove that for $\tau$ small,
  $\V_\fui(x)=\arg\max (\eta^\tau*\varphi)(x+\tau q)-\tau L(x, q)$ is
  a singleton that defines a $C_1^2$ map.
        To do so, we define the map $\Phi:[0,+\infty)\times\T^d\times\R^d\to\R^d$  by
\begin{equation*}
\Phi(\tau,x,q)=D(\eta^\tau*\fui)(x+\tau q)-D_qL(x,q).
\end{equation*}
Then,
\begin{align}\nonumber
&\Phi(0,x, D_pH(x,D\fui(x)))=0,\\  \nonumber 
  &\Phi_q(\tau,x,q)=
\tau D^2(\eta^\tau*\fui)(x+\tau q)-D^2_{qq}L(x,q).
\end{align}
	Since $D^2(\eta^\tau*\fui)(x) =(\eta^\tau*D^2\fui)(x)$ and
        $L(x,q)$ satisfies (A1), there is $s_\fui>0$ 
	such that $\Phi_q(\tau,x,q)$ is strictly negative definite for
        all $\tau<s_\fui$ and $(x,q)\in\T^d\times\R^d$. Therefore, the map
       $q\mapsto\eta^\tau*\fui(x+\tau q)-\tau L(x,q)$ is concave for
       any $\tau<s_\fui$, $x\in\T^d$. Thus, for any $\tau<s_\fui$, $x\in\T^d$,
       the solution $q=V_\fui(\tau,x)$ to $\Phi(\tau,x,q)=0$ is the  maximizer of this map.
The compactness of $\T^d$ and the implicit function
        theorem imply  that there is $0<t_\fui\leq s_\fui$ such
     that $V_\fui\in C_1^2 ([0,t_\fui)\times\T^d,\R^d)$ and $V_\fui(0, x)=D_pH(x,D\fui(x))$.
        Furthermore, from the definition of the Legendre transform, we have
        $V_\fui(\tau, x)=D_pH(x,D(\eta^\tau*\fui)(x+\tau V_\fui(\tau, x))).$
\end{proof}
\begin{corollary}\label{HJ-disc-cont}
 Assume $L$ satisfies (A1).
  Let $\fui\in C^2(\T^d)$, then
	\begin{equation*}
	\lim_{\tau\to 0}
	\left\|\frac{\cL_\tau\fui(x)-\fui(x)}{\tau}
	-\Delta\fui(x)-H(x,D\fui)\right\|_{C^0}=0.
	\end{equation*}
\end{corollary}
\begin{proof}
	Let $V\in C_1^2([0,t_\fui)\times\T^d,\R^d)$ be the map given in
	Proposition \ref{implicit} and $R=\max|V_\fui|$. Let $\om:
	[0,+\infty)\to [0,+\infty)$ be the function given by Lemma
	\ref{discrete-continuous}. Then  
	\begin{equation}\label{aprox-disc-cont}
	\left|\frac{\cL_\tau\fui(x)-\fui(x)}{\tau}-\Delta\fui(x)-\lip
	D\fui,V_\fui(\tau,x)\rip+L(x,V_\fui(\tau,x))\right|\le\om(\tau).
	\end{equation}
	Since $V_\fui(0,x)=D_pH(x,D\fui(x))$
	\begin{multline}\label{aprox-H}
	\lip D\fui,V_\fui(\tau,x)\rip-L(x,V_\fui(\tau,x))-H(x,D\fui(x))=\\
	\lip D\fui(x),V_\fui(\tau,x)-V_\fui(0,x))\rip - L(x,V_\fui(\tau,x))+L(x,V_\fui(0,x)).
	\end{multline}
	From \eqref{aprox-disc-cont},  \eqref{aprox-H} the Corollary follows.
\end{proof}

\begin{lemma}\label{semiconvex}
	Let $u\in C(\T^d)$ and suppose that $L$ satisfies (A1) and (A2). Then, $\cL_\tau  u$ is semi-convex.
\end{lemma}
\begin{proof}
	Fix $x\in\T^d$  and let $q\in\R^d$ be such that
	$\cL_\tau  u(x) = - \tau L(x, q) +(\eta^\tau*u)(x+\tau
        q)\big)$. Using (A2), for 
	$h\in\R^d$, we have
	\begin{multline*}
	\cL_\tau  u(x+h)-2\cL_\tau  u(x)+\cL_\tau  u(x+h)\\\ge
-\tau L\Big(x+h, q-\frac h\tau\Big)+2\tau L(x, q)-\tau L\Big(x-h, q+\frac{h}{\tau}\Big)
\\	\ge -\tau k_1|h|^2-k_2\frac{|h|^2}\tau=-\Big(\tau k_1+\frac{k_2}\tau\Big)|h|^2.\qedhere
	\end{multline*}
\end{proof}

\subsection{The nonlocal case}
\label{sec:nonlocal}
Here, we examined 
the discrete MFG  \eqref{eq:disc-MFG}
with a nonlocal $\bF$ and prove Theorem \ref{nonlocal}. In the following Lemma, we establish a fundamental estimate, the uniform semiconvexity of the approximations. 
\begin{lemma}\label{uni-semicon}
Under the hypothesis of Proposition  \ref{bound} and assuming that $L$
satisfies $(A2)$ and that $\bF$ satisfies $(B3)$, let $(\rho_\tau,u_\tau,\tilde\mu_\tau)$ solve \eqref{eq:disc-MFG}. Then, 
	$(u_\tau)_{0<\tau<1}$ is uniformly semi-convex; that is, there is a
	constant, $C>0$, independent of $\tau$, such that
\begin{equation}\label{uni-semicon-u}
u_\tau(x+h)-2u_\tau(x)+u_\tau(x-h)\ge -C|h|^2,
\end{equation}
	for any $x,h\in\R^d$, $0<\tau<1$.
\end{lemma}
\begin{proof} Fix $0<\tau<1$, 
  and to simplify the notation, we denote the solution to \eqref{eq:disc-MFG} by $(\rho, u, \tilde\mu)$.
	Since $u=\cL_\tau u-\tau\bF(x,\mu)-\tau \rho$, Lemma
	\ref{semiconvex} and (B3) imply that  $u$ is semi-convex.
	To find the constant $C$ in \eqref{uni-semicon-u}, we define
        inductively a sequence $\{\Lam_n\}^\infty_{n=0}$ of
        semi-convexity moduli of $u$, taking the modulus $\Lam_0$
        such that $\Lam_0 ^2>(k_2+\Lam_0)(k_1+k_0)$. Assuming we have
        chosen $\Lam_n$ such that
		\[u(x+h)-2u(x)+u(x-h)\ge -\Lam_n|h|^2,\]
we have
	\begin{align}\notag
(\eta^\tau*u)(x+h)-&(\eta^\tau*u)(x)+(\eta^\tau*u)(x-h)
\\\notag&=\int_{\T^d}\eta^\tau(y)\left(u(x+h-y)-2u(x-y)+u(x-h-y)\right)dy\\
\label{eq-conv-conv}
&\ge -\int_{\T^d}\eta^\tau(y) \Lam_n|h|^2\ dy= -\Lam_n|h|^2.
	\end{align}
	For $x\in\T^d$, let $q\in\R^d$ be such that
	$u(x) = - \tau L(x, q) -\tau\bF(x,\mu)+(\eta^\tau*u)(x+\tau q)\big)-\tau\rho$.  
	Using \eqref{eq-conv-conv} and (B2), for $h\in\R^d$ and $0\le\te\le 1$, we get
	\begin{align*}
	u(x+h)&-2u(x)+u(x-h)
	\ge-\tau\Bigl[L(x+h,q-\te\frac h\tau)-2L(x,q)+L(x-h,q+\te\frac h\tau)\Bigr]\\
	&-\tau k_0|h|^2+(\eta^\tau\ast u)(x+\tau q+(1-\te)h)-2(\eta^\tau \ast u)(x+\tau q)
	\\&+(\eta^\tau\ast u)(x+\tau q-(1-\te)h)\\
	&\ge   -\tau (k_1+k_0)|h|^2-k_2\frac{\te^2|h|^2}\tau-\Lam_n(1-\te)^2|h|^2.
	\end{align*}
	Optimizing over $\te$, we obtain $\Lam_{n+1}$ as
	\[\Lam_{n+1}=\tau (k_1+k_0)+\frac {k_2\Lam_n}{k_2+\Lam_n\tau}.\]
        Recalling that $\tau<1$ and $\Lam_0$ satisfies
        $\Lam_0 ^2>(k_2+\Lam_0)(k_1+k_0)$, we have
        \[\Lam_1=\frac{\tau (k_1+k_0)(k_2+\Lam_0\tau)+k_2\Lam_0}{k_2+\Lam_0\tau}
          <\frac{\tau\Lam_0 ^2+k_2\Lam_0}{k_2+\Lam_0\tau}=\Lam_0.\]
        Since
        \[\Lam_{n+1}-\Lam_n=\frac {k_2\Lam_n}{k_2+\Lam_n\tau}
          -\frac {k_2\Lam_{n-1}}{k_2+\Lam_{n-1}\tau}
          =\frac{k_2^2(\Lam_n-\Lam_{n-1})}{(k_2+\Lam_n\tau)(k_2+\Lam_{n-1}\tau)},\]
we get by induction that $\{\Lam_n\}$ is a decreasing sequence.
        In the limit, this procedure gives a universal modulus
	\[\Lam(\tau)=\frac{(k_1+k_0)\tau+\sqrt{(k_1+k_0)^2\tau^2+4k_1k_2}}2.\]
    Finally, we take $C=\Lam(1)$.
\end{proof}
 A uniformly semi-convex  family of functions on
$\T^d$ is uniformly Lipschitz. Thus, by the preceding result, we have
that $(u_\tau)_\tau$ is a uniformly Lipschitz family of functions. 

\begin{proof}[Proof of Theorem \ref{nonlocal}]
From Proposition \ref{unique}, we know that the pair $(\rho_\tau, \mu_\tau)$ is unique.
	Let $\tau_n\to 0^+$. From Remark \ref{rel-comp}, by passing to a 
	sub-sequence, we can assume that $(\rho_{\tau_n},
	\tilde\mu_{\tau_n})$ converges to  $(\rho, \tilde\mu)\in\R\times\cP_\ell$,
	and from Proposition 4.2 in \cite{DIPPS}, $\tilde\mu\in\cC$.
	Then $\mu_{\tau_n}\rightharpoonup\tilde\mu=\Pr_{1\#}
        \tilde\mu$. Since $\bF$ is uniformly continuous,
	$\bF(x,\mu_{\tau_n})$ converges uniformly to  $\bF(x,\mu)$ and so
	\begin{equation*}
	\lim_{n\to\infty}\int_{\T^d}\bF(x,\mu_{\tau_n})\ d\mu_{\tau_n}(x)
	=\int_{\T^d}\bF(x,\mu)\ d\mu(x). 	
	\end{equation*}	
	Thus,
	\begin{align}\notag
	-\rho
	=
	\lim_{n\to +\infty} -\rho_{\tau_n}
	&= \lim_{n\to +\infty} \int_{\T^d\times\R^d}L(x,q)+\bF(x,\mu_{\tau_n})\ 
	d\tilde\mu_{\tau_n}(x,q)\\\label{rho-rho_0}
	&= \int_{\T^d\times\R^d} L(x,q)+\bF(x,\mu)\ d\tilde\mu(x,q).
	\end{align}
 Because $(u_\tau)_\tau$ is uniformly Lipschitz, we can extract a
subsequence if necessary, and therefore assume that  $(u_{\tau_n})_n$ converges uniformly to some function
$u\in C(\T^d)$. We claim that   $u$ solves
	\begin{equation}
	\label{eq:mfg1}
	\Delta v +H(x,Dv(x))=\rho+\bF(x,\mu)
	\end{equation}
	in the viscosity sense. This will be established by 
 proving that $u$ is both a viscosity sub-solution and a viscosity super-solution.

First, we prove that $u$ is a sub-solution of \eqref{eq:mfg1}.  
For that, let $\fui\in C^2(\T^d)$
and suppose that
 $x_0\in \operatorname{argmax}(u-\fui)$ is
a strict maximum point. 
Accordingly, we can find a sequence
$(x_n)_n$
of points of maximum of  $u_{\tau_n}-\fui$
that converges to $x_0$ in $\T^d$. 
Let $\ep_n:=\max(u_{\tau_n}-\fui)$  
and define $\fui_n:=\fui+\ep_n$. 
By construction, we have 
\[u_{\tau_n}\leq \fui_n\hbox{ in }\T^d\]
and, in addition, $u_{\tau_n}(x_n)=\fui_n(x_n)$.


Because $\cL_\tau u\geq \cL_\tau v$ if $u\geq v$ , we conclude that 
\begin{align*} \fui_n(x_n)&=u_{\tau_n}(x_n)
= \cL_{\tau_n}u_{\tau_n}(x_n)-\tau_n\rho_{\tau_n}-\tau_n\bF(x,\mu_n)\\&
\leq
\cL_{\tau_n}\fui_{n}(x_n)-\tau_n\rho_{\tau_n}-\tau_n\bF(x,\mu_n).
\end{align*}
Therefore, from the identity $\fui_n=\fui+\ep_n$, we obtain
%
%
%
\begin{equation*}
\frac{\cL_{\tau_n}\fui(x_n)-\fui(x_n)}{\tau_n} \geq  \rho_{\tau_n}+\bF(x,\mu_n).
\end{equation*}

For each $n\in\N$, we select $q_n\in\R^d$ such that
$\cL_{\tau_n}=\left(\eta^{\tau_n}*\fui\right)(x_n+\tau_n
q_n)-\tau_nL(x_n,q_n)$. Accordingly,  the preceding inequality becomes
\begin{equation}\label{eq2 asymptotic}
\frac{\left(\eta^{\tau_n}*\fui\right)(x_n+\tau_n q_n)-\fui(x_n)}{\tau_n}-L(x_n,q_n)\geq \rho_{\tau_n}+\bF(x,\mu_n). 
\end{equation}
Combining 
the estimate
$|D(\eta^\tau*\fui)(x_n)|\leq \|D\fui\|_\infty$
with 
fact (iii) in subsection \ref{sec:statement-problem}, we conclude that there  
exists $R>0$ such that $\|q_n\|\le R$ for every $n\in\N$. 
Thus, by
extracting a further subsequence if necessary, we have that $q_n$
converges to some
$q\in \R^d$.  
By considering the limit $n\to +\infty$ in  \eqref{eq2 asymptotic}, we obtain
\[
\Delta \fui(x_0)+\langle D\fui(x_0),q\rangle-L(x_0,q) \geq \rho +\bF(x_0,\mu).
\]
Thus, 
\[
\Delta \fui(x_0)+H(x_0, D\fui(x_0)) \geq \rho +\bF(x_0,\mu), 
\]
Therefore, $u$ is a viscosity sub-solution to \eqref{eq:mfg1}.

Now, we prove that 
 $u$ is a viscosity super-solution of \eqref{eq:mfg1}. For that, 
fix $\fui\in C^2(\T^d)$ 
and suppose that
$x_0\in \operatorname{argmin}(u-\fui)$ is
a strict minimum point. 
Because $x_0$ is a strict minimum of $u-\fui$, we can find 
 a sequence
	$(x_n)_n$ converging to $x_0$ in $\T^d$ such that $x_n\in \operatorname{argmin} (u_{\tau_n}-\fui)$. Let $\ep_n:=\min(u_{\tau_n}-\fui)$  
	and set  $\fui_n:=\fui+\ep_n$. 
	Arguing as in the first part, we end up with 	
%
%
	\begin{equation*}
	\frac{\cL_{\tau_n}\fui(x_n)-\fui(x_n)}{\tau_n} \leq \rho_{\tau_n}+\bF(x,\mu_n).
	\end{equation*}

	By the definition of $\cL_{\tau_n}$, for every fixed $q\in\R^d$, we have
   \[\frac{\left(\eta^{\tau_n}*\fui\right)(x_n+\tau_n q)-\fui(x_n)}{\tau_n}-L(x_n,q) \leq  
	\rho_{\tau_n}+\bF(x,\mu_n).\]
	Thus, considering the limit $n\to +\infty$ and using Proposition
        \ref{discrete-continuous}, we conclude that 
	\[
	\Delta \fui(x_0)+\langle D\fui(x_0),q\rangle-L(x_0,q) \leq \rho+\bF(x_0,\mu).
	\]
	Finally, taking the supremum with respect to
	$q\in\R^d$ in
	the prior inequality,  we get 
	\[
	\Delta\fui(x_0)+H(x_0,D\fui(x_0))\leq \rho+\bF(x_0,\mu), 
	\]
	using the definition of the Legendre transform.
Consequently, $u$ is a viscosity super-solution to \eqref{eq:mfg1}.
	
Applying Theorem \ref{teo stochastic Mather problem} to the  Lagrangian
	$L(x,q)+\bF(x,\mu)$, we conclude that \eqref{rho-rho_0} implies that 
	\[\tilde\mu=G_{V\#}\nu,\] where $\nu\in\cP(\T^d)$ solves \eqref{eq FK} 
	with $V(x):=D_pH (x,Du(x))$.  Therefore, $\mu=\nu$,
	and so $(\rho, u,\mu)$ solves the ergodic MFG \eqref{eq:estacionario}.
\end{proof}

\subsection{Bounds for the local case} \label{bound.loc}  In this part, we assume that $L$ satisfies $(A3)$ and $F$ satisfies $(B6)$.
Let $(\rho_\tau, u_\tau, \tilde\mu_\tau)$ solve the
$\tau$--discrete MFG and set $\mu_\tau =\Pr_{1\#}\tilde\mu_\tau$. According to
Remark \ref{rel-comp}, we have that
\[\int_{\T^d}F(\mu_\tau)\ d\mu_\tau=-\rho_\tau-\int_{\T^d\times\R^d}L\ 
d\tilde\mu_\tau,\]
is uniformly bounded in $\tau$.Therefore, for $F(m)=m^a$, $0<a<1$, we have 
$F(m)m=F(m)^{1+1/a}$. Consequently, the  $L^{1+1/a}$ norm of
$F(\mu_\tau)$ is uniformly bounded in $\tau$.

Define the operators 
\begin{align*}
\Ga_\tau[u]&:= \eta^\tau\ast u,\\
\cN_\tau[u](x)&:=\max_{v\in\R^d} \left[u(x+\tau v)- \tau K(v)\right].
\end{align*}
Then, $\cL_\tau u=\cN_\tau\circ\Ga_\tau[u]+\tau U(x)$.
First, we observe that for any $\tau>0$, $x\in\T^d$
\[u(x)\le  \cN_\tau[u](x)\le \max u.\]
Moreover, if $u(x_0)=\max u$, we have $\cN_\tau[u](x_0)\ge u(x_0)$
and so \[\cN_\tau[u](x_0)=\max u.\] 
Because $\eta^\tau$ is a probability
density and $s\to |s|^r$ is convex for $r>1$, we have
\begin{align*}
|\eta^\tau*u(x)|^r&=\left|\int_{\T^d} \eta^\tau(y-x)u(y)\ dy\right|^r\le 
\int_{\T^d} \eta^\tau(y-x)|u(y)|^r\ dy.
\end{align*}
Consequently, integrating the foregoing expression, we have
\begin{align*}
\int_{\T^d}|\eta^\tau*u(x)|^r\ dx&\le \int_{\T^d}\int_{\T^d} \eta^\tau(y-x)|u(y)|^r\ dy\ dx\\&=
\int_{\T^d}\int_{\T^d} \eta^\tau(y-x)|u(y)|^r\ dx\ dy=\int_{\T^d} |u(y)|^r\ dy.
\end{align*}
Thus, for any $r\geq 1$, we have
\begin{equation}
\label{eq:Ga}
\|\Ga_\tau[u]\|_{L^r}\le \|u\|_{L^r}.
\end{equation}
Define
\begin{equation}
\label{eq:def-f}
f^\tau:=\frac{\cN_\tau\circ\Ga_\tau[u_\tau]-u_\tau}\tau.
\end{equation}
Then, we have $f_\tau=\rho_\tau+F(\mu_\tau)-U$. Thus,
$\|f_\tau\|_{L^{1+1/a}}$ is bounded uniformly in $\tau$. 
Note that \eqref{eq:def-f} does not depend on the choice of the
solution $u_\tau$ as solutions are unique up to the addition of constants.
Let $v_\tau=u_\tau-\max \Ga_\tau[u_\tau]$, then
$\max\Ga_\tau[v_\tau]=0$. Thus,  we have
\begin{equation}
  \label{eq:fixpoint}
v_\tau=-\tau f^\tau+\cN_\tau\circ\Ga_\tau[v_\tau]
\end{equation}
Observe that $w(t,x)=\cN_t[u](x)$ solves the initial value problem
\begin{equation*}
\begin{cases}
w_t-K(Dw)=0\\  
w(0,x)=u(x)
\end{cases}.
\end{equation*}
We have that $\max\Ga_\tau[v_\tau]=0$. Additionally, by  performing a translation, we
can assume that $\Ga_\tau[v_\tau](0)=0$.
Letting $W(t,x)=-\cN_t\circ\Ga_\tau[v_\tau](x)$, we have $W(t,x)\ge 0=W(t,0)$.
We deduce
\begin{align}\notag
\frac{d}{dt} \int_{\T^d}W^r &=-r\int_{\T^d}W^{r-1}K(-DW)\le-cr\int_{\T^d}W^{r-1}|DW|^q \\
&=-cr\left( \frac{q}{r+q-1}\right)^q\int_{\T^d}\left|DW^{\frac{r+q-1}{q}}\right| ^q. \label{eq-main}
\end{align}
\begin{proposition}\label{pro-Poin}
  Let $q>d$, $f\in W^{1,q}(\T^d)$ and suppose that $f(0)=0$. Then, there exists a positive constant, $C$, such that
	\begin{equation}\label{eq-Poin}
\int_{\T^d} |f|^q\,dx\leq C \int_{\T^d} |Df|^q\,dx.
	\end{equation}
\end{proposition}
\begin{proof}
Suppose that the inequality in \eqref{eq-Poin} is not true. Then, there
exists a sequence $\{f^n\}_{n=1}^\infty\subset W^{1,q}(\T^d)$ satisfying $f^n(0)=0$,
$||f^n||_{L_q(\T^d)}=1$ and
\begin{equation}\label{eq-proof}
||Df^n||_{L_q(\T^d)}\leq \frac{1}{n}.
\end{equation}By Morrey theorem, we have that $W^{1,q}(\T^d)\subset C^{0,\gamma}(\T^d)$, 
for some $0<\gamma<1$; that is,
\begin{equation*}
||\varphi||_{C^{0,\gamma}(\T^d)}\leq C ||\varphi||_{W^{1,q}(\T^d)},\quad \varphi\in W^{1,q}(\T^d).
\end{equation*}
Therefore, $\{f^n\}_{n=1}^\infty\subset C^{0,\gamma}(\T^d)$. The sequence  $\{f^n\}$ is equicontinuous and 
$f^n(0)=0$. Hence, b the Arzela-Ascoli theorem,  there is a subsequence $f^{n_k}$
that converges uniformly to a function  $\bar{f}$ with $\bar{f}(0)=0$.
The sequence $|f^{n_k}|^2$ converges uniformly to $|\bar{f}|^2$. Thus,
\begin{equation}
  \label{eq-L2}
  \int_{\T^d}|\bar{f}|^2=1.
\end{equation}
On the other hand,  from \eqref{eq-proof}, we have that $D\bar{f}=0$. Hence, $\bar f$ is constant and, thus,  identically $0$. This contradicts \eqref{eq-L2}.


\end{proof}
From \eqref{eq-main}, Proposition \ref{pro-Poin}, and $\|w\|_{L^r}\le \|w\|_{L^{r+q-1}}$,
we obtain a constant $C_r>0$ such that
\[\frac{d}{dt} \int_{\T^d}W^r~dx\le -r C_r \int_{\T^d}W^{r+q-1}~dx\le -\frac {rC_r}{q-1}
\Bigl(\int_{\T^d}W^r~dx\Bigr)^\frac {r+q-1}r.\]
Denoting $h(t)=\int_{\T^d}W^r~dx$, from the preceding estimate,  we have
\begin{equation*}
\frac{h^\prime(t)}{h^{\frac {r+q-1}r}(t)}\leq-\frac {rC_r}{q-1}.
\end{equation*}
Integrating the previous inequality over $[0,\tau]$, we obtain
\begin{equation}\label{key}
-\frac {q-1}r\left( h^{\frac {-q+1}r}(\tau)-h^{\frac {-q+1}r}(0)\right)\leq -\frac {rC_r}{q-1}.
\end{equation} 
 Taking into account the definitions of $h$ and $W$, from \eqref{key}, we obtain the estimate
\begin{equation}\label{key1}
\|\cN_\tau\circ\Ga_\tau[v_\tau]\|^{q-1}_{L^r}\leq\frac{\|\Ga_\tau[v_\tau]\|^{q-1}_{L^r} }
{1+C_r\tau\|\Ga_\tau[v_\tau]\|^{q-1}_{L^r} }.
\end{equation}
Because the function $s\mapsto s/(1+\tau C_rs)$ is increasing, from \eqref {eq:Ga} we have
\[\frac{\|\Ga_\tau[v_\tau]\|^{q-1}_{L^r} }
{1+C_r\tau\|\Ga_\tau[v_\tau]\|^{q-1}_{L^r} }\le
\frac{\|v_\tau\|^{q-1}_{L^r(\T^d)}}{1+\tau C_r \|v_\tau\|^{q-1}_{L^r(\T^d)}}.\]
Using the preceding inequality in \eqref{key1}, 
 from \eqref{eq:fixpoint}, we deduce that
\[ \|v_\tau\|_{L^r(\T)}\le \tau\|f^\tau\|_{L^r(\T)}+\frac{\|v_\tau\|_{L^r(\T)}}{1+\tau
    C_r\|v_\tau\|_{L^r(\T)}},\]
for $d=1$ and $q=2$.  Therefore,
\begin{equation*}\label{eq:lr-bd}
   \|v_\tau\|_{L^r(\T)}\leq\frac{\tau C_r\|f^\tau\|_{L^r(\T)}+
\sqrt{( \tau C_r\|f^\tau\|_{L^r(\T)})^2+4 C_r\|f^\tau\|_{L^r(\T)}}}{2C_r}.
\end{equation*}
For $d>2$, defining $p$ by $\dfrac 1{q-1}+\dfrac 1p=1$, from
\eqref{eq:fixpoint}, Minkowsky and H\"older inequalities, we have
\begin{align*}
  \|v_\tau\|^{q-1}_{L^r(\T^d)}
& \le 
 \big (\tau\|f^\tau\|_{L^r(\T^d)}+\|\cN_\tau\circ\Ga_\tau[v_\tau]\|_{L^r(\T^d)}\big)^{q-1}\\&\le \big(\tau\|f^\tau\|_{L^r(\T^d)}^{q-1}+\|\cN_\tau\circ\Ga_\tau[v_\tau]\|^{q-1}_{L^r(\T^d)}\big)(\tau+1)^\frac{q-1}p\\&\le \left (\tau\|f^\tau\|^{q-1}_{L^r(\T^d)}+\frac{\|v_\tau\|^{q-1}_{L^r(\T^d)}}{1+\tau
  C_r\|v_\tau\|^{q-1}_{L^r(\T^d)}}\right) (\tau+1)^\frac{q-1}p,
  \end{align*}
which implies
\begin{align}\notag
   \|v_\tau\|^{q-1}_{L^r(\T^d)}&\leq\frac\tau 2 (\tau+1)^\frac{q-1}p\|f^\tau\|^{q-1}_{L^r(\T^d)}+
     \frac{(\tau+1)^\frac{q-1}p-1}{2C_r\tau}\\\label{eq:lr-bound}&+
    \sqrt{\Bigl( \frac\tau 2  (\tau+1)^\frac{q-1}p\|f^\tau\|^{q-1}_{L^r(\T^d)}+
                    \frac{(\tau+1)^\frac{q-1}p-1}{2C_r\tau}\Bigr)^2
+\frac{ (\tau+1)^\frac{q-1}p\|f^\tau\|^{q-1}_{L^r(\T^d)}}{C_r}}
\end{align}
      Thus, recalling that $\|f_\tau\|_{L^{1+1/a}}$ is  uniformly bounded in $\tau$, we deduce that  for $r=1+1/a$, $\|v_\tau\|_{L^r(\T^d)}$ is bounded uniformly in $\tau\in(0,1)$.
Using $\eqref{eq:const-bound}$, $\max\cN_\tau\circ\Ga[v_\tau]=0$ and
$F(m)\ge 0$, we also have
\[v_\tau=\tau(U-\rho_\tau-F\circ\mu_\tau)+\cN_\tau\circ\Ga[v_\tau]\le 
\tau(\max U-\min U+F(1)).\]

\subsection{The weak solutions for local coupling}
\label{sec:weak-approach}
In this subsection, using the monotonicity properties of the MFG, 
we define weak solutions (see Definition \ref{weak}) for the ergodic MFG system
\eqref{eq:estacionario} and for the $\tau$-discrete MFG problem
\eqref{eq:disc-MFG} in the local case.  Then, we 
verify that solutions of the $\tau$-discrete MFG problem
\eqref{eq:disc-MFG} are also weak solutions. Next, relying on the
a priori estimate of the previous subsection and 
using Minty's method (see
\cite{Eva}), 
we conclude that our approximations (up to a normalization)
converge weakly to a weak solution to  \eqref{eq:estacionario}. 
Finally, using an unpublished result due to Vardan Voskanyan, 
we prove that any weak solution of \eqref{eq:estacionario} is 
a classical solution.
Therefore, normalized solutions of the $\tau$-discrete MFG system
\eqref{eq:disc-MFG} weakly converge to a classical solution of the MFG system \eqref{eq:estacionario}.

Let  $\cH^\tau: C(\T^d)\to C(\T^d)$ be
\[\cH_\tau u(x)=(\cL_\tau u(x)-u(x))/\tau.\]
Next, we prove several  properties of $\cH_\tau$, which are crucial for our study of the weak solutions to the MFG system
\eqref{eq:estacionario} in the sense of Definition \ref{weak}.
\begin{proposition}\label{pro-Hh} Fix $x\in\T^d$. Then, the map
  $u\mapsto  \cH_\tau u(x)$  is convex. Moreover,  let $\partial_u\cH_\tau(x)$ be
  the  subdifferential of $\cH_\tau$ at $u$.
  Consider the set 
 \[ \V_u(x):=\arg\max[(\eta^\tau *u)(x+\tau q)-\tau L(x,q)].\]
 Then, for $w\in C(\T^d)$ and $v\in   \V_u(x)$, 
  the functional $w\mapsto((\eta^\tau *w)(x+\tau v)-w(x))/\tau$
  belongs to $\partial_u\cH_\tau(x)$.
\end{proposition}
\begin{proof}The convexity of  $\cH_\tau$ follows from the  inequality
	\begin{align*}
\tau \cH_\tau(\lam u_1+(1-\lam)u_2)(x)=&
\max_q\lam(\eta^\tau*u_1(x+\tau q)-u_1(x)-\tau L(x,q))\\+&(1-\lam)(\eta^\tau*u_2(x+\tau q)-u_2(x)-\tau L(x,q))\\
\le&\lam \tau \cH_\tau(u_1)+(1-\lam) \tau \cH_\tau(u_2),
\end{align*}
for any $u_1, u_2\in C(\T^d)$ and $0\leq \lambda\leq 1$. 
Note that for any $w\in C(\T^d)$, we have 
\begin{align}\label{DH1}
	\cL_\tau (u+w)(x)&\ge(\eta^\tau*(u+w))(x+\tau v)-L(x,v)\\&=
	\cL_\tau u(x)+(\eta^\tau *w)(x+\tau v),\nonumber
\end{align}
because $v\in   \V_u(x)$. 
Therefore,
	\begin{align} \label{DH2}
	\tau \cH_\tau (u+w)(x)&\ge \tau \cH_\tau (u)(x) +(\eta^\tau *w)(x+\tau v)-w(x).
\end{align}
Thus, the linear map $w\mapsto((\eta^\tau *w)(x+\tau v)-w(x))/\tau$
belongs to $\partial_u\cH_\tau(x)$.
\end{proof}

For $u\in C(\T^d)$ consider a Borel measurable map $V:\T^d\to\R^d$ such
that $V(x)\in \V_u(x)$ for all $x\in\T^d$.  For each
$x\in\T^d$, the linear map $\zeta_V(x):C(\T^d)\to\R $, defined by
$\zeta_V(x)w=((\eta^\tau *w)(x+\tau V(x))-w(x))/\tau$ belongs to $\partial_u\cH_\tau(x)$.
For $m\in\cP(\T^d)$, we define $\zeta_V^ *m\in C (\T^d)^ *$ by
\[\lip\zeta_V^ *m,w\rip=\int_{\T^d} \zeta_V(x)w\ dm(x),\]
and let
\[\partial_u\cH_\tau^*m =\{\zeta_V^ *m\mid V : \T^d \to\R^d\hbox{ Borel
measurable, } \forall x\in\T^d,~\ V(x)\in\V_u(x)\}.\]
Next, we
define the multivalued operator $A^\tau:D(A^\tau)\subset \R\times L^1(\T^d)\times L^1(\T^d)\to\R\times L^1(\T^d)\times L^1(\T^d)$
\begin{equation}\label{A-tau-op}
A^\tau\begin{bmatrix}\rho\\u\\m\end{bmatrix}
=\begin{bmatrix}1-\int_{\T^d} dm
	\\\partial_u\cH_\tau^*m\\ -\cH_\tau u +F(m)+\rho
\end{bmatrix},\quad (\rho,u,m)\in D(A^\tau),
\end{equation}
where $D(A^\tau)=\R\times C(\T^d)\times \{m\in L^1(\T^d):m\geq 0\} $.

Relying on  Proposition \ref{pro-Hh}, we prove the monotonicity  of $A^\tau$, which is  crucial for defining weak solutions.
\begin{proposition} Suppose that  Assumption (B4) holds and that $(\rho_1,u_1,m_1), (\rho_2,u_2,m_2) \in D(A^\tau)$. Let  $(\sigma_1,\nu_1,v_1)\in A^\tau(\rho_1,u_1,m_1)$ and $(\sigma_2,\nu_2,v_2)\in A^\tau(\rho_2,u_2,m_2)$. Then,  

\begin{equation}\label{mono-def}
\bigg\lip \begin{bmatrix}\sigma_1-\sigma_2\\\nu_1-\nu_2\\ v_1-v_2\end{bmatrix},
\begin{bmatrix}\rho_1-\rho_2\\u_1-u_2\\m_1-m_2\end{bmatrix}\bigg\rip_{((\R\times L^1(\T^d)\times L^1(\T^d))^*, D(A^\tau))}\geq 0.
\end{equation}
\end{proposition}
\begin{proof} The proof results from the following computations. For $i=1,2$, we have that $(\sigma_i,\nu_i,v_i)\in A^\tau(\rho_i,u_i,m_i)$,
   $\nu _i=\zeta_{V_i}^ *m_i$,
$V _i : T^d \to \R^d$ Borel measurable, $V_i(x)\in\V_{u_i}(x)$. Therefore, recalling that the map
$u\mapsto  \cH_\tau u(x)$  is convex (see Proposition \ref{pro-Hh}) and taking into account that $F$ is increasing (see Assumption (B4)), we deduce 
 \begin{align*} 
	\bigg\lip
  \begin{bmatrix}\sigma_1-\sigma_2\\\nu_1-\nu_2\\v_1-v_2\end{bmatrix},
	\begin{bmatrix}\rho_1-\rho_2\\u_1-u_2\\m_1-m_2\end{bmatrix}\bigg\rip
	&=\int_{\T^d} (\cH_\tau(u_2)-\cH_\tau(u_1)~d(m_1-m_2)\\&+\int_{\T^d}(F(m_1)-F(m_2)+\rho_1-\rho_2) ~d(m_1-m_2)\\
	&+\int_{\T^d} (u_1-u_2)\ d(\nu_1-\nu_2)
	+(\rho_1-\rho_2)\int_{\T^d}d(m_2-m_1) \\
	&=\int_{\T^d} (\cH_\tau(u_2)(x)-\cH_\tau(u_1)(x)+\zeta_{V_1}(x)(u_1-u_2))~dm_1\\
	&+\int_{\T^d} (\cH_\tau(u_1)(x)-\cH_\tau(u_2)(x)+\zeta_{V_2}(x)(u_2-u_1))~dm_2
	\\&+  \int_{\T^d} (F(m_1)-F(m_2)) ~d(m_1-m_2)\geq 0.\qedhere
\end{align*}
\end{proof}
Note that by the monotonicity of the operator $A^\tau$, we mean that $A^\tau$ satisfies \eqref{mono-def}.
Setting
\[A^0\begin{bmatrix}\rho\\u\\m\end{bmatrix}
=\begin{bmatrix}
1-\int_{\T^d} dm \\\Delta m-\diver(mD_pH(x,Du))\\-\Delta u-H(x,Du)+F(m)+\rho\
\end{bmatrix},\]
and using \eqref{A-tau-op}, we can write the MFG system \eqref{eq:estacionario} and the $\tau$-discrete MFG system \eqref{eq:disc-MFG}  as 
\begin{equation}\label{MFG-OP}
A^\tau\begin{bmatrix}\rho\\u\\m\end{bmatrix}
=0,\quad\tau\geq 0.
\end{equation}
Now, using the concept  of the weak solutions  for the monotone MFGs developed in the series of papers \cite{FG2,FGT1,FeGoTa20,DRT2021Potential}, we define the weak solutions to \eqref{MFG-OP}.

\begin{definition}\label{weak}
  We say that $(\rho,u,\mu)\in \R\times L^1(\T^d)\times L^1(\T^d)$
  with $\mu\geq 0$
  is a weak solution to \eqref{MFG-OP} if there is  $k>1$ such that 
  $u\in  L^k(\T^d)$ and
  for any $m,\fui\in C^2(\T^d)$, $\lam\in\R$ with $m\geq 0$,  and any
  $(\sigma,\nu,v)
        \in A^\tau(\lam,\fui, m)$, we have
 
\[\bigg 	\lip	\begin{bmatrix}\sigma\\\nu\\v\end{bmatrix},\begin{bmatrix}\lam-\rho\\\fui-u\\m-\mu\end{bmatrix}\bigg\rip_{((
	\R\times L^k(\T^d)\times L^1(\T^d))^*, \,\R\times L^k(\T^d)\times L^1(\T^d))}\ge 0.\]
\end{definition}
The following proposition verifies that solutions of $\tau$-discrete MFG system \eqref{eq:disc-MFG} 
are weak solutions to \eqref{MFG-OP} for $\tau>0$.
\begin{proposition}\label{discrete-weak} Suppose that $L$ satisfies
  $(A1)$ and $F$ satisfies $(B4)$ and $(B5)$. Then, a solution 
  $(\rho_\tau,u_\tau,\tilde\mu_\tau)\in \R\times C(\T^d) \times\cC_\tau$ to \eqref{eq:disc-MFG}  with $\tau>0$,
  is a weak solution to \eqref{MFG-OP}.
\end{proposition}
\begin{proof} For simplicity, we  omit the subscript $\tau$ for $(\rho_\tau,u_\tau,\tilde\mu_\tau)$.
		Let $\fui, m\in C^2(\T^d)$, $\lam\in\R$. 
	Since $\tilde\mu$ is supported in 
	$\{(x,q):q\in\V_u(x)\}$, we have from \eqref{DH1} and \eqref{DH2} that 
	\begin{equation*}
	\begin{split}
	\int_{\T^d} \cH_\tau (\fui)\ d\mu&\ge\int_{\T^d} \cH_\tau (u)\ d\mu
\\&+\int_{\T^d} \frac{(\eta^\tau *(\fui-u))(x+\tau q)-(\fui-u)(x)}\tau\ d\tilde\mu=
\int_{\T^d} \cH_\tau (u)\ d\mu.
	\end{split}
	\end{equation*}
  For $(\sigma,\nu,v)
  \in A^\tau(\lam,\fui,m)$, $\nu=\zeta_V^ *m $,   $V : \T^d \to \R^d$  Borel measurable,
 $V(x)\in\V_\fui(x)$, we have
	\begin{align*}
	\bigg\lip \begin{bmatrix}\sigma\\\nu\\v\end{bmatrix},
	\begin{bmatrix}\lam-\rho\\\fui-u\\m-\mu\end{bmatrix}\bigg\rip
	&=\int_{\T^d} (-\cH_\tau\fui+F(m)+\lam)\ d(m-\mu)
	\\&+\int_{\T^d}\zeta_V (\fui-u)\ dm+(\lam-\rho)\int_{\T^d} d(\mu-m)\\
	&=\int_{\T^d} (\zeta_V(\fui-u)-\cH_\tau\fui)\ dm+
	\int_{\T^d} \cH_\tau\fui\ d\mu\\&+\int_{\T^d} F(m)\ d(m-\mu) -\rho\int_{\T^d} d(\mu-m)\\
&	\ge \int_{\T^d} (-\cH_\tau u+F(m)+\rho)\ d(m-\mu)\\&
	=\int_{\T^d} (F(m)-F(\mu))\ d(m-\mu)\ge 0.\qedhere
	\end{align*}
\end{proof}
We aim to prove the weak convergence of solutions to
\eqref{eq:disc-MFG}. For that purpose, we need the following proposition. 
\begin{proposition}\label{op-conv}
Suppose that $L$ satisfies $(A1)$ and let $\fui, m\in C^2(\T^d)$, $\rho\in\R$. Then,
	\[\lim_{\tau\to 0}A^\tau\begin{bmatrix}\rho\\\fui\\m\end{bmatrix}
	=A^0\begin{bmatrix}\rho\\\fui\\m\end{bmatrix}.\]
\end{proposition}
\begin{proof}
	According to Corollary \ref{HJ-disc-cont}, we have
	\[\lim_{\tau\to 0} -\cH_\tau\fui+F(m)+\rho=-\Delta\fui-H(x,D\fui)+F(m)+\rho,
	\]
	where the convergence is uniform in $x$.
	Let $t_\fui$ and $V_\fui$ be given by Proposition
        \ref{implicit} and $R=\max|V_\fui|$. 
        
        Next, we  compute explicitly $\partial_\fui \cH_\tau^*$.
        For $(\tau,x)\in(0,t_\fui)\times\T^d$,
	we  have
	\begin{equation*}
		\partial_\fui \cH_\tau^*m=\zeta^*_{V_\fui}m,
	\end{equation*}
and 
	\[\lip \partial_\fui \cH_\tau^*m, w\rip= \int_{\T^d} \zeta_{V_\fui}(x)w\ dm(x)
	=\int_{\T^d}\frac{\eta^\tau*w(x+\tau V_\fui(\tau,x))-w(x)}\tau m(x)\ dx.\]
	Note that
	\begin{align*}
	\cI:=\int_{\T^d}\eta^\tau*w(x+\tau V_\fui(\tau,x))m(x)\ dx&=\int_{\T^d}\int_{\T^d}\eta^\tau(y)w(x+\tau V_\fui(\tau,x)-y)\ dy\ m(x)\ dx\\
&=\int_{\T^d}\eta^\tau(y) \int_{\T^d} w(x+\tau V_\fui(\tau,x)-y) m(x)\ dx\ dy.
	\end{align*}
	Let $\Phi_\tau=\Phi(\tau,\cdot)$ be the inverse map of $I+\tau V_\fui(\tau, )$. 
	The implicit function theorem and the compactness of $\T^d$ imply that there is
	$\tau_\fui$ such that  $\Phi_\tau$ is well defined for $\tau<\tau_\fui$  and
	$\Phi$ is smooth in $\tau$. Therefore, we can change variables in $\cI$ such that $z=x+\tau V_\fui(\tau,x)-y$ and $x=\Phi_\tau(y+z)$. Hence,
	\begin{align*}
		\cI&=\int_{\T^d}\eta^\tau(y) \int_{\T^d} w(z) m(\Phi_\tau(y+z))\det D\Phi_\tau(y+z)\ dz\ dy\\&=
		\int_{\T^d} w(z)\int_{\T^d}\eta^\tau(y) m(\Phi_\tau(y+z))\det D\Phi_\tau(y+z)\ dy\ dz\\
		&=\int_{\T^d} w(z)\eta^\tau*(m\circ \Phi_\tau\det D\Phi_\tau)(z) dz.
	\end{align*}
	Setting $\Psi(\tau,z)=\eta^\tau*(m\circ \Phi_\tau\det D\Phi_\tau)(z)$, we notice that  $\Psi(0,z)=m(z)$ and
	\begin{align*}
		\partial_\fui \cH_\tau^*m(z)&=\frac{\Psi(\tau,z)-\Psi(0,z)}\tau=
		\int_0^1D_1\Psi(s\tau,z)\ ds.
	\end{align*}
Consequently,
	\begin{align*}
		\lim_{\tau \to 0}\partial_\fui \cH_\tau^*m(z)&=\frac {\partial\Psi}{\partial\tau}(0,z).
	\end{align*}
	On the other hand, taking $W_\tau$ such that $\Phi_\tau=I-\tau W_\tau$, we observe that  
	\begin{equation*}
x+\tau (V_\fui(\tau, x)-W_\tau( x+ \tau V_\fui(\tau, x)))=x.
	\end{equation*}
Therefore,
	\[W_\tau( x+\tau V_\fui(\tau, x))=V_\fui(\tau, x)=D_pH(x,D(\eta^\tau*\fui)(x+\tau
	V_\fui(\tau, x))),\]
which implies
	\begin{align*}  \frac {\partial\Psi}{\partial\tau}&=
		\Delta\Psi+
		\eta^\tau*\Bigg(Dm\circ\Phi_\tau\cdot \frac{d\Phi_\tau}{d\tau}\det  D\Phi_\tau+
		m\circ \Phi_\tau\frac {d\det D\Phi_\tau}{d\tau} 
		\Bigg ),	\end{align*}  
	and
	\begin{align*} 
		\frac {\partial\Psi}{\partial\tau}(0,z)
		&=\Delta m-Dm\cdot W_0-m\diver W_0=\Delta m-\diver(m D_pH(x,D\fui)).\qedhere
	\end{align*}  
\end{proof}

The next result of this subsection  is the following  weak-strong uniqueness result.  If there exists a
classical solution to \eqref{MFG-OP} with $\tau=0$,  then any weak
solution agrees with this solution. This result was first proven by Vardan Voskanyan
but was never published. We present here the proof for completeness. 
\begin{lemma}\label{unique1} Suppose that $L$ satisfies $(A1)$ and $F$
  is strictly increasing.
	Let $k>1$ and $(r,u,m)\in\R \times L^k(\T^d)\times L^1(\T^d)$  be a weak solution (in the sense of Definition \ref{weak}) of \eqref{MFG-OP} with $\tau=0$, and let $(\rho,\fui,\mu)\in\R \times C^2(\T^d)$
	be a classical solution. Then, $(r,m)=(\rho, \mu)$ and $u-\fui$ is constant.
\end{lemma}
\begin{proof}
	For any $\nu,\psi\in\C^2(\T^d)$  and $\lam\in\R$, the function  
	\[h(\ep)=
	\bigg\lip A^0\begin{bmatrix}\rho+\ep\lam\\\fui+\ep\psi\\\mu+\ep\nu\end{bmatrix},
	\begin{bmatrix}\rho+\ep\lam-r\\\fui+\ep\psi-u\\\mu+\ep\nu-m\end{bmatrix}\bigg\rip\]
	has a minimum at $\ep= 0$, so $h'(0)=0$. Thus, taking into account that
	$(\rho,\fui,\mu)$
	is a classical solution of the equation involving $A^0$, we obtain 
	\begin{equation}\label{last-eq}
\begin{split}
		\bigg\lip DA^0\begin{bmatrix}\rho\\\fui\\\mu\end{bmatrix}
\begin{pmatrix}\lam\\\psi\\\nu\end{pmatrix},
\begin{bmatrix}\rho-r\\\fui-u\\\mu-m\end{bmatrix}\bigg\rip&
=\int_{\T^d}(-\Delta\psi-D_pH(x,D\fui)D\psi+F'(\mu)\nu+\lam)d(\mu-m)\\&
+\int_{\T^d}(\Delta\nu-\diver(\nu D_pH(x,D\fui)d(\mu-m)
-(\rho-r)\int_{\T^d} d\nu\\
&+\int_{\T^d}\mu D^2_{pp}H(x,D\fui)D\psi)(\fui-u)~dx=0.
\end{split}
	\end{equation}
	Letting $\lam=\rho-r$, and using a density argument, we take
	$\psi=\fui-u$, $\nu=\mu-m$  such that
	\[
	\int_{\T^d} F'(\mu)(\mu-m)^2+\mu \lip D^2_{pp}H(x,D\fui)D(\fui-u),D(\fui-u)\rip=0.
	\]
Hence, 	relying on the facts that $\mu, F'(\mu)>0$ and $D^2_{pp}H$ is positive definite, we get
	$m=\mu$ and $Du=D\fui$. Using these in \eqref{last-eq}, we obtain $r=\rho$.
\end{proof}
 Now, using the previous results, we prove Theorem \ref{local}.
\begin{proof}[Proof of Theorem \ref{local}]
By Proposition \ref{discrete-weak},  $(\rho_\tau,u_\tau,\tilde\mu_\tau)\in\R\times C(\T^d) \times\cC_\tau$ is a weak
solution to \eqref{MFG-OP} with $\tau>0$.
From \eqref{eq:lr-bound}, we have that $\|u_\tau\|_{L^{1+1/a}(\T^d)}$ is
uniformly bounded in $\tau\in(0,1)$.  Let $\tau_n>0$ be a sequence
converging to zero. By Lemma \ref{op-conv}, there is a subsequence, still denoted $\tau_n$, such that
$u_{\tau_n}$ converges weakly in $L^{1+1/a}(\T^d)$ to $u$. By Remark \ref{rel-comp}, there is a further
subsequence, still denoted $\tau_n$, such that $(\rho_{\tau_n},\mu_{\tau_n})$ 
converges in $\R\times\cP_\ell$ to $(r, m)$. Thus, $(r,u,m)$
is a weak solution to
\eqref{MFG-OP} with $\tau=0$. Therefore, by Lemma \ref{unique1}, we conclude the proof.
\end{proof}

\bibliographystyle{plain}

\bibliography{mfg.bib}

%
%
%

\end{document}